\documentclass{amsart}

\usepackage[title]{appendix}
\usepackage{amssymb}
\usepackage{amsmath}
\usepackage{amsthm}
\usepackage{tikz-cd}
\usepackage{mathtools}
\usepackage{framed}
\usepackage{comment}

\definecolor{darkblue}{rgb}{0,0,0.6}
\usepackage[breaklinks, ocgcolorlinks,colorlinks=true, citecolor=darkblue, filecolor=darkblue, linkcolor=darkblue, urlcolor=darkblue]{hyperref}

\usepackage[margin=30mm]{geometry}

\newcommand{\C}{\mathbb{C}}

\newcommand{\N}{\mathbb{N}}
\newcommand{\Q}{\mathbb{Q}}
\newcommand{\R}{\mathbb{R}}

\newcommand{\Z}{\mathbb{Z}}
\renewcommand{\H}{\mathbb{H}}
\newcommand{\F}{\mathbb{F}}

\renewcommand{\mod}{\,\,\text{\normalfont mod}\,}

\DeclareMathOperator{\Aut}{Aut}

\DeclareMathOperator{\Res}{Res}

\DeclareMathOperator{\PSL}{PSL}

\DeclareMathOperator{\Def}{def}
\DeclareMathOperator{\IM}{Im}
\DeclareMathOperator{\Tor}{Tor}
\DeclareMathOperator{\Ker}{Ker}

\DeclareMathOperator{\SL}{SL}
\DeclareMathOperator{\TL}{TL}

\DeclareMathOperator{\Cls}{Cls}
\DeclareMathOperator{\mass}{mass}

\DeclareMathOperator{\disc}{disc}
\DeclareMathOperator{\Ram}{Ram}

\DeclareMathOperator{\Gal}{Gal}

\DeclareMathOperator{\ord}{ord}
\DeclareMathOperator{\ei}{ei}
\DeclareMathOperator{\bm}{\mathcal{B}_{\text{max}}}
\DeclareMathOperator{\SF}{SF}
\DeclareMathOperator{\HT}{HT}
\DeclareMathOperator{\PHT}{PHT}

\DeclareMathOperator{\MP}{MP}
\DeclareMathOperator{\std}{std}

\newtheorem{thmx}{Theorem}

\newtheorem{corx}[thmx]{Corollary}

\newtheorem{thm}{Theorem}[section]

\newtheorem*{thm*}{Theorem}
\newtheorem{prop}[thm]{Proposition}
\newtheorem*{prop*}{Proposition}
\newtheorem{lemma}[thm]{Lemma}
\newtheorem{corollary}[thm]{Corollary}
\newtheorem*{corollary*}{Corollary}

\newtheorem{problem}[thm]{Problem}

\theoremstyle{definition}

\newtheorem*{question*}{Question}

\newtheorem*{theorem*}{Theorem}

\theoremstyle{remark}
\newtheorem{remark}[thm]{Remark}
\newtheorem*{remark*}{Remark}

\usepackage{rotating}

\usepackage{color}
\usepackage{enumerate}

\usepackage{cite}
\usepackage[nodisplayskipstretch]{setspace}
\usepackage{multirow}
\usepackage{placeins}
\usepackage{mathrsfs}
\usepackage{mathabx}
\usepackage{adjustbox}

\usepackage{esvect}

\newenvironment{clist}[1]
{\begin{enumerate}[\normalfont #1]}
{\end{enumerate}}

\newcommand{\wh}{\widehat}
\newcommand{\wt}{\widetilde}

\makeatletter
\@namedef{subjclassname@2020}{%
  \textup{2020} Mathematics Subject Classification}
\makeatother

\begin{document}

\title{Cancellation for $(G,n)$-complexes and the Swan finiteness obstruction}

\author{John Nicholson}
\address{School of Mathematics and Statistics, University of Glasgow, United Kingdom}
\email{john.nicholson@glasgow.ac.uk}



\begin{abstract} 
In previous work, we related homotopy types of finite $(G,n)$-complexes when $G$ has periodic cohomology to projective $\Z G$-modules representing the Swan finiteness obstruction. 
We use this to determine when $X \vee S^n \simeq Y \vee S^n$ implies $X \simeq Y$ for finite $(G,n)$-complexes $X$ and $Y$, and give lower bounds on the number of homotopically distinct pairs when this fails.
The proof involves constructing projective $\Z G$-modules as lifts of locally free modules over orders in products of quaternion algebras, whose existence follows from the Eichler mass formula.
In the case $n=2$, difficulties arise which lead to a new approach to finding a counterexample to Wall's D2 problem.
\end{abstract}

\maketitle


\section{Introduction}

For an integer $n \ge 2$, a \textit{$(G,n)$-complex} is a connected $n$-dimensional CW-complex $X$ with fundamental group $G$ and whose universal cover $\wt X$ is $(n-1)$-connected. 
Equivalently, it is the $n$-skeleton of a $K(G,1)$. Given a group $G$ and $n \ge 2$, a finite $(G,n)$-complex exists if and only if $G$ has type $F_n$ in the sense of Wall \cite{Wa65}.
For a group $G$ and $n \ge 2$, let $\HT(G,n)$ be the set of homotopy types of finite $(G,n)$-complexes which is a graded tree with edges between each $X$ and $X \vee S^n$ and with grading coming from the directed Euler characteristic which is defined to be $\vv \chi (X) = (-1)^n \chi(X)$.
The first basic question is whether or not $\HT(G,n)$ has cancellation, i.e. has the property that $X \vee S^n \simeq Y \vee S^n$ implies that $X \simeq Y$ for all $X, Y \in \HT(G,n)$. This is equivalent to the existence of a finite $(G,n)$-complex $X_0$ such that every finite $(G,n)$-complex is of the form $X_0 \vee rS^n$ for some $r \ge 0$.

This question is completely inaccessible in general.
However, a solution in the case where $G$ is finite abelian, which includes non-cancellation examples, follows from work of Browning \cite{Br79b}, Dyer-Sieradski \cite{DS73} and Metzler \cite{Me76}.
Examples of non-cancellation are also known over certain infinite groups, with the first coming from work of Dunwoody \cite{Du76}.
These examples are of special interest due to their applications to smooth $4$-manifolds \cite{BC+21, HK93-survey, KS84}, Wall's D2 problem \cite{Jo03a, Ni19} and combinatorial group theory \cite{MR17}.

The aim of this article is to study the cancellation problem over groups with periodic cohomology. This is the second part of a two-part series in which part one \cite{Ni20a} reduces the classification of $\HT(G,n)$ to problems about projective $\Z G$-modules which we resolve in the present article.

\subsection{Main results}

Let $\PHT(G,n)$ denote the set of polarised homotopy types of finite $(G,n)$-complexes, i.e. the homotopy types of pairs $(X, \rho)$ where $X$ is a finite $(G,n)$-complex and $\rho : \pi_1(X) \to G$ is an isomorphism.
Recall that a D2 \textit{complex} is a connected CW-complex which is cohomologically 2-dimensional and a group $G$ has the D2 \textit{property} if every finite D2 complex $X$ with $\pi_1(X) \cong G$ is homotopy equivalent to a $2$-complex \cite{Wa65}. 
We say that a non-trivial finite group $G$ has \textit{$k$-periodic cohomology} if its Tate cohomology groups satisfy $\widehat{H}^{i}(G;\Z) \cong \widehat{H}^{i+k}(G;\Z)$ for all $i \in \Z$.

Our main result is that cancellation for $\HT(G,n)$ is completely determined by $m_{\H}(G)$. 
\begin{thmx} \label{thm:main}
Let $G$ have $k$-periodic cohomology, let $n = ik$ or $ik-2$ for some $i \ge 1$ and, if $n=2$, suppose $G$ has the {\normalfont D2} property. Then the following are equivalent:
\begin{enumerate}[\normalfont(i)]
\item $\HT(G,n)$ has cancellation
\item $\PHT(G,n)$ has cancellation
\item $m_{\H}(G) \le 2$.
\end{enumerate}
\end{thmx}

Here $m_{\H}(G)$ denotes the number of copies of the quaternions $\H=M_1(\H)$ in the Wedderburn decomposition of $\R G$ for a finite group $G$ (see \cite[p318]{Ni18}).

\begin{remark} \label{remark:thm-main}
(a) The equivalence of (ii) and (iii) in the case $n=2$ was previously established in \cite[Theorem A]{Ni19}, building upon earlier work of F. E. A. Johnson \cite{Jo03a}. The equivalence of (i) and (ii) is in spite of the fact that, in \cite[Section 9]{Ni20a}, we showed that there exists $G$ with $4$-periodic cohomology such that $\HT(G,n) \ne \PHT(G,n)$ for $n>2$ is even.

(b) 
A statement like Theorem \ref{thm:main} can still be made in the case $n=2$ without the hypothesis that $G$ satisfies the D2 property. In this case, $\HT(G,2)$ (resp. $\PHT(G,2)$) should be replaced by the set of homotopy (resp. polarised homotopy) types of finite D2 complexes $X$ with $\pi_1(X) \cong G$. 
See Theorem \ref{thm:main-D2} for the precise statement as well as an outline of the proof.
\end{remark}

Note that, if $G$ has $k$-periodic cohomology, then $k$ is even (see, for example, \cite[Proposition 40.2]{Jo03a}) and so $n$ is always even in the statement above.
Recall that, if $G$ is finite and $n$ is even, then $\HT(G,n)$ is a \textit{fork} in the sense that it has a single vertex at each non-minimal grade $\vv \chi(X)$ and finitely many at the minimal level \cite[Corollary 4.7]{Ni20a}. Let $N(G,n) = \#\{X \in \HT(G,n) : \vv \chi(X) \text{ is minimal}\}$. 

Our second result is the following. This complements Theorem \ref{thm:main} since $\HT(G,n)$ has cancellation if and only if $N(G,n)=1$. In particular, Theorem \ref{thm:main} says that $N(G,n) > 1$ if and only if $m_{\H}(G) \ge 3$, and Theorem \ref{thm:main-bounds} demonstrates the extent to which cancellation fails as $m_{\H}(G) \to \infty$.

\begin{thmx} \label{thm:main-bounds}
Let $G$ have $k$-periodic cohomology, let $n = ik$ or $ik-2$ for some $i \ge 1$ and, if $n=2$, suppose $G$ has the {\normalfont D2} property. If $m = m_{\H}(G)$, then
\[ N(G,n) \ge e^{\tfrac{m \log m}{8 \log \log m} + O(m \log \log m)}. \]
\end{thmx}

\begin{remark} \label{remark:thm-main-bounds}
As in Remark \ref{remark:thm-main}, the statement holds in the case $n=2$ without the assumption that $G$ has the D2 property with $N(G,2)$ replaced by the number of homotopy types of finite D2 complexes $X$ with $\pi_1(X) \cong G$ and $\chi(X)$ ($=\vv\chi(X)$) minimal.
\end{remark}

The proofs of Theorems \ref{thm:main} and \ref{thm:main-bounds} will be based on results in \cite{Ni20a} which reduce the problems to pure algebra.
Let $n = ik$ or $ik-2$, let $P_{(G,n)}$ be a projective $\Z G$-module representing the Swan finiteness obstruction $\sigma_{ik}(G)$ and, if $n=2$, suppose that $G$ has the D2 property. Then, by \cite[Theorems A, B]{Ni20a}, there are isomorphisms of graded trees
\[\Psi: \PHT(G,n) \to [P_{(G,n)}], \quad \bar{\Psi}: \HT(G,n) \to [P_{(G,n)}]/\Aut(G)\]
where $[P_{(G,n)}]$ is the set of projective $\Z G$-modules $P$ with $P \oplus \Z G^i \cong P_{(G,n)} \oplus \Z G^j$ for some $i,j \ge 0$.
We view this as a graded tree by adding edges between $P$ and $P \oplus \Z G$ for all $P \in [P_{(G,n)}]$.

We will now define the action of $\Aut(G)$ on $[P_{(G,n)}]$. First recall that, if $\theta \in \Aut(G)$ and $P$ is a projective $\Z G$-module, then we can define $P_\theta$ to be the $\Z G$-module with the same underlying abelian group as $P$ but whose action is given by $g \cdot m = \theta(g) \cdot_P m$ for $g \in G$ and $m \in P$. Note that this does not necessarily give an action of $\Aut(G)$ on $[P_{(G,n)}]$ since we may have that $P \in [P_{(G,n)}]$ but $P_\theta \not \in [P_{(G,n)}]$.
Recall that, for $G$ a finite group, every projective $\Z G$-module is of the form $P \oplus \Z G^r$ where $P$ has rank one and $r \ge 0$ \cite{Sw60b}.
If $I \subseteq \Z G$ is the augmentation ideal and $r$ is an integer such that $(r,|G|)=1$, then $(I,r)$ is a projective $\Z G$-module \cite{Sw60a}. 
If $G$ has $k$-periodic cohomology, then the action of $\theta \in \Aut(G)$ on $[P_{(G,n)}]$ is then given by 
\[\theta : P \oplus \Z G^r \mapsto ((I, \psi_k(\theta)^i) \otimes P_\theta) \oplus \Z G^r\] 
where $P$ has rank one and $\psi_k : \Aut(G) \to (\Z/|G|)^\times$ is a map which depends only on $G$ \cite[Section 7]{Ni20a}. 

In particular, in order to prove Theorem \ref{thm:main}, it suffices to determine when $[P_{(G,n)}]$ and $[P_{(G,n)}]/\Aut(G)$ have cancellation for $G$ a group with $k$-periodic cohomology and an appropriate integer $k$. In order to prove Theorem \ref{thm:main-bounds}, it suffices to compute lower bounds on the number of equivalence classes of rank one projective modules in $[P_{(G,n)}]/\Aut(G)$.
These tasks will be achieved by combining detailed information on the modules $P_{(G,n)}$ with the cancellation theorems for projective $\Z G$-modules recently established by the author in \cite{Ni18} and \cite{Ni19}.

\subsection{Applications to Wall's D2 problem}
\label{ss:D2}

We now consider the case $n=2$ in more detail. Recall that every finite presentation $\mathcal{P}$ for a group $G$ has an associated presentation complex $X_{\mathcal{P}}$ which is a finite 2-complex with $\pi_1(X_{\mathcal{P}}) \cong G$. Conversely, every (connected) finite 2-complex $X$ with $\pi_1(X) \cong G$ is homotopy equivalent to $X_{\mathcal{P}}$ for some finite presentation $\mathcal{P}$ for $G$. 
We say that two finite presentations $\mathcal{P}$, $\mathcal{Q}$ are homotopy equivalent if $X_{\mathcal{P}}$, $X_{\mathcal{Q}}$ are homotopy equivalent (see Section \ref{section:D2} for more details).

For a finite presentation $\mathcal{P}$ define its \textit{deficiency} $\Def(\mathcal{P})$ to be the number of generators minus the number of relators. For a finitely presented group $G$, define the \textit{deficiency} $\Def(G)$ to be the maximum deficiency across all finite presentations for $G$.
Note that $\chi(X_{\mathcal{P}}) = 1 - \Def(\mathcal{P})$ and so $N(G,2)$ is the number of homotopy classes of presentations of $G$ with maximal deficiency.

Let $G$ be a group with 4-periodic cohomology.
It is a consequence Theorems \ref{thm:main} and \ref{thm:main-bounds} that, if $G$ has the D2 property, then we would expect non-cancellation examples for finite 2-complexes over $G$ provided $m_{\H}(G)$ is sufficiently large. For example, this applies when $G=Q_{28}$ is the quaternion group of order 28 since $G$ has the D2 property and $m_{\H}(G)=3$ \cite[Theorem 8.11]{Ni19}.

However, if we are interested in the D2 property itself, we could instead view Theorems \ref{thm:main} and \ref{thm:main-bounds} as a constraint that needs to be satisfied in order for the D2 property to hold. 
For example, let $Q_{4n}$ denote the quaternion group of order $4n$ which has standard presentation $\mathcal{P}_{\std} = \langle x , y \mid x^n = y^2, y^{-1}xy = x^{-1} \rangle$. Since $\Def(Q_{4n})=0$, the presentations $\mathcal{P}$ with maximal deficiency are \textit{balanced} in that $\Def(\mathcal{P})=0$.
If $Q_{4n}$ has the D2 property for all $n \ge 2$, then $m_{\H}(Q_{4n}) = \lfloor n/2 \rfloor$ implies that $N(Q_{4n},2) \gg e^{\lambda n}$ for all $\lambda >0$ and $n$ sufficiently large.
In \cite[Section 3]{MP19}, Mannan-Popiel introduced a family of presentations 
\[ \mathcal{P}_{\MP}^r = \langle x , y \mid x^n = y^2, y^{-1}xyx^{r-1} = x^ry^{-1}x^2y \rangle  \]
which, for certain $r \in \Z$, present $Q_{4n}$. Note that, if $r=0$ or $1$, then $\mathcal{P}_{\MP}^r$ is homotopy equivalent to $\mathcal{P}_{\std}$.
It was proposed in \cite{MP19} that the $\mathcal{P}_{\MP}^r$ contain all presentations for $Q_{4n}$ up to homotopy. 
In Proposition \ref{prop:MP}, we will show that $r \equiv s \mod n$ implies $\mathcal{P}_{\MP}^r \simeq \mathcal{P}_{\MP}^s$ and so there are at most $n$ presentations of the form $\mathcal{P}_{\MP}^r$ up to homotopy. Hence, for $n$ sufficiently large, either $Q_{4n}$ does not have the D2 property or $Q_{4n}$ has a presentation which is not of the form $\mathcal{P}_{\MP}^r$ up to homotopy.

This has the following broad generalisation.
For a finite presentation $\mathcal{P} = \langle x_1, \cdots, x_n \mid r_1, \cdots, r_m \rangle$, define the \textit{relation length} $\ell(\mathcal{P})$ to be $\sum_{i=1}^m |r_i|$ where $|\cdot |$ is the word norm on the free group $F(x_1,\cdots,x_n)$, i.e. $|1| = 0$ and, for $r \ne 1$, $|r|$ is the minimal $t \ge 1$ for which $r = s_1 \cdots s_t$ for some $s_i \in \{x_1^{\pm 1}, \cdots x_n^{\pm 1}\}$.
For the presentations for $Q_{4n}$ above, we have $\ell(\mathcal{P}_{\std}) = n+6$ and $\ell(\mathcal{P}_{\MP}^r) = n+2r+8 \le 3n+6$ for $2 \le r < n$.
The following is a consequence of the fact that, if $Q_{4n}$ has the D2 property, then $N(Q_{4n},2)$ grows super-exponentially in $n$.

\begin{corx} \label{cor:main-D2}
Let $\lambda > 0$. Then, for all $n$ sufficiently large, at least one of the following holds:
\begin{clist}{(i)}
\item 
$Q_{4n}$ does not have the {\normalfont D2} property.
\item 
$Q_{4n}$ has a balanced presentation which is not homotopy equivalent to a two-generator two-relator presentation $\mathcal{P}$ with $\ell(\mathcal{P}) \le \lambda n$.
\end{clist}
\end{corx}

\begin{remark} \label{remark:cor-main-D2}
Whilst we have restricted to $Q_{4n}$ for simplicity, an analogous result holds with $Q_{4n}$ replaced by an arbitrary group with periodic cohomology. Furthermore, in condition (ii), we can replace $\ell(\mathcal{P}) \le \lambda n$ with a slightly weaker bound.
For a more detailed statement, see Theorem \ref{thm:main-D2-detailed}.
\end{remark}	
	
This suggests that, if $Q_{4n}$ has the D2 property for all $n$, then there would need to exist a family of presentations for $Q_{4n}$ for which the maximal relation length growths super-linearly in $n$.
This leads to a new approach to finding a counterexample to the D2 problem; namely, by obstructing the existence of such a large number of presentations for $Q_{4n}$ which are distinct up to homotopy.

This would be in contrast to the situation for finite abelian groups $G$ where, from results of Metzler \cite{Me76} and Browning \cite{Br79b}, we know that every presentation for $G$ with maximal deficiency is homotopy equivalent to a presentation $\mathcal{P}$ with $\ell(\mathcal{P}) < \ell(\mathcal{P}_{\std}) + t$ where $t \le |G|$ is the greatest common divisor of the invariant factors of $G$ and $\mathcal{P}_{\std}$ denotes the standard presentation for $G$. 

\subsection{Applications to the classification of 4-manifolds}

In \cite[Section 1.3]{Ni20a}, we discussed how a key motivation behind the homotopy classification of finite $(G,n)$-complexes is the homotopy classification of $2n$-dimensional manifolds. Each finite $(G,n)$-complex $X$ has an associated closed smooth $2n$-manifold $M(X)$ which is well-defined up to homotopy equivalence \cite{KS84}. For fixed $G$ and $n$, the manifolds $M(X)$ are \textit{stably diffeomorphic} in the sense that, for finite $(G,n)$-complexes $X$, $Y$, there exists $r \ge 0$ such that $M(X) \# r(S^n \times S^n) \cong M(Y) \# r(S^n \times S^n)$ are diffeomorphic.
The case $n=2$ is of particular interest since the only known examples of closed smooth $4$-manifolds which are stably diffeomorphic but not homotopy equivalent are of the form $M(X)$, $M(Y)$ where $X$, $Y$ are homotopically distinct finite 2-complexes with certain finite abelian fundamental groups \cite{KS84}. 

Let $G$ be a finite group with $4$-periodic cohomology and $m_{\H}(G) \ge 3$. Then, conditional on $G$ having the D2 property, Theorem \ref{thm:main} implies that there exists homotopically distinct finite 2-complexes $X_i$ for $i=1,2$ with $\pi_1(X_i) \cong G$ and $\chi(X_1)=\chi(X_2)$. This raises the question of whether the $M(X_i)$ contain further examples of stably diffeomorphic closed smooth 4-manifolds which are homotopically distinct.

\begin{problem} \label{problem:smooth-man}
Let $G$ have $4$-periodic cohomology with $m_{\H}(G) \ge 3$.
Do there exist homotopically distinct finite $2$-complexes $X$, $Y$ with fundamental group $G$ and $\chi(X)=\chi(Y)$ such that $M(X)$, $M(Y)$ are homotopically distinct?
\end{problem}

One difficulty with tackling this problem is that, if $G$ has $4$-periodic cohomology and $m_{\H}(G) \ge 3$ then we do not know whether homotopically distinct finite 2-complexes $X$, $Y$ with fundamental group $G$ and $\chi(X)=\chi(Y)$ actually exist unless we (somehow) know that $G$ has the D2 property.

In general, we can still apply Theorem \ref{thm:main-D2} (see also Remark \ref{remark:thm-main} (b)) which is the same statement as Theorem \ref{thm:main} but with $\HT(G,2)$ replaced by the set of homotopy types of finite D2-complexes $X$ with $\pi_1(X) \cong G$. 
Recently, Adem-Hambleton showed that every finite D2-complex $X$ with $\pi_1(X) \cong G$ finite has an associated closed topological 4-manifold  \cite[Theorem 3.8]{AH23}. It may therefore be possible to use Theorem \ref{thm:main-D2} to obtain new examples of closed topological 4-manifolds which are stably homeomorphic but not homotopy equivalent. Such examples are known to exist by results of Freedman \cite{Fr82} though by a different construction.
We will not explore Problem \ref{problem:smooth-man} further in this article.

\subsection{The action of $\Aut(G)$ on general projective modules} \label{ss:special-properties}

We will conclude by considering two special properties that the projective $\Z G$-modules $P_{(G,n)}$ have, and ask whether these properties are shared by arbitrary projective $\Z G$-modules. This is of purely algebraic interest.

In order to compute $\HT(G,n)$ from $\PHT(G,n) \cong [P_{(G,n)}]$, we needed to quotient by a certain action of $\Aut(G)$ on $[P_{(G,n)}]$ which was defined using $\psi: \Aut(G) \to (\Z/|G|)^\times$.
Let $C(\Z G)$ denote the projective class group and $T_G$ the Swan subgroup which is generated by the $(I,r)$ where $(r,|G|)=1$. The existence of the action was dependent on the following property of $P_{(G,n)}$:
\begin{enumerate}
\item[(P1)] If $\theta \in \Aut(G)$, then $[(P_{(G,n)})_\theta] = [P_{(G,n)}] \in C(\Z G)/T_G$, i.e. $[P_{(G,n)}] \in (C(\Z G)/T_G)^{\Aut(G)}$.
\end{enumerate}

Given such an action, the second property that the $P_{(G,n)}$  have is a consequence of Theorem \ref{thm:main}:
\begin{enumerate}
\item[(P2)] $[P_{(G,n)}]$ has cancellation if and only if $[P_{(G,n)}]/ \Aut(G)$ has cancellation.
\end{enumerate}

These properties each hold for free modules since the $\Aut(G)$ action on a free module is trivial. This implies that the class $\Z G \in [\Z G]$ is a fixed point under the $\Aut(G)$-action and so $[\Z G]$ has cancellation if and only if $[\Z G]/\Aut(G)$ has cancellation.

In contrast to this, we show that properties (P1) and (P2) fail for arbitrary projective $\Z G$-modules, even in the case where $G$ has periodic cohomology. For (P1), we show that, if $G = C_p$ for $p \ge 23$ prime, then there exists $[P] \in C(\Z G)$ and $\theta \in \Aut(G)$ such that $[P_\theta] \ne [P] \in C(\Z G)/T_G$ (see Theorem \ref{thm:example1}).
For (P2) note that, if $G = Q_{28}$ is the quaternion group of order 28, then $T_G = 0$ and $\Aut(G)$ acts trivially on $C(\Z G)$. We will prove the following, which is stated later as Theorem \ref{thm:main-example}.

\begin{thm} \label{thm:main-example-intro}
There exists a projective $\Z Q_{28}$-module $P$ such that $[P]$ has non-cancellation and $[P]/\Aut(Q_{28})$ has cancellation. Here $\theta \in \Aut(Q_{28})$ acts on $[P]$ by sending $P_0 \mapsto (P_0)_\theta$ where $P_0 \in [P]$.
\end{thm}

\subsection{Organisation}

The paper will be structured as follows.
In Section \ref{section:group-theory}, we begin by establishing the necessary group-theoretic facts on groups with periodic cohomology. This includes calculating $m_{\H}(G)$ for each group and relating its value to the vanishing of $\sigma_k(G)$. 
In Section \ref{section:quotients}, we establish preliminaries on induced representations and show that, if $f: G \twoheadrightarrow Q_{4n}$ then there is an induced map $f_* : \Aut(G) \to \Aut(Q_{4n})$ and a surjection $[P_{(G,n)}]/ \Aut(G) \twoheadrightarrow [\widebar{P_{(G,n)}}]/\IM(f_*)$. This allows us to show non-cancellation occurs for $G$ by considering the case $Q_{4n}$. 

In Section \ref{section:cancellation-periodic}, we will combine the results in Section \ref{section:group-theory} with \cite[Theorem 5.1]{Ni19} to show that, if $m_{\H}(G) \le 2$, then $[P_{(G,n)}]$ has cancellation. Since the converse also holds, this leads to a complete determination of when cancellation occurs for a representative of $\sigma_k(G)$ and implies (ii) $\Leftrightarrow$ (iii) in Theorem \ref{thm:main}.

In Section \ref{section:mass-formulas}, we discuss locally free modules over orders in quaternion algebras and the Eichler mass formula. 
In Section \ref{section:orders-in-QG} we study the orders $\Lambda = \Lambda_{n_1, \cdots, n_k}$ in  $A=\prod_{i=1}^{k} \Q[\zeta_{n_i},j]$ which arise as quotients of $\Z Q_{4n}$. 
 In Theorem \ref{thm:higher-orders}, we classify the $\Lambda$ for which every stably free $\Lambda$-module is free, completing the classification started by Swan in \cite[Section 8]{Sw83}.  
 In Sections \ref{section:non-cancellation} and \ref{section:bounds}, we then apply these results to prove Theorems \ref{thm:main} and \ref{thm:main-bounds}.
 
In Section \ref{section:D2}, we discuss applications to the classification of group presentations, culminating in a proof of Corollary \ref{cor:main-D2}. In Section \ref{section:examples}, we explore the special properties of the modules $P_{(G,n)}$, as described in Section \ref{ss:special-properties}.

This work can be viewed as an attempt to properly amalgamate the techniques and results obtained by Swan in \cite{Sw83} with the wider literature on applications of the Swan finiteness obstruction \cite{Jo03a, Mi85, Ni19}. As such, we will rely heavily on calculations from \cite{Sw83} but will give alternate proofs where possible. 

\subsection*{Acknowledgements}
I would like to thank my PhD supervisor F. E. A. Johnson whose suggestion that I read `Swan's long paper' \cite{Sw83} was largely the inspiration for this work. I would also like to thank the referees for a number of helpful comments which improved the exposition of this article. This work was supported by the UK Engineering and Physical Sciences Research Council
 (EPSRC) grant EP/N509577/1 as well as the Heilbronn Institute for Mathematical Research. 

\section{Groups with periodic cohomology} \label{section:group-theory}

Recall that a \textit{binary polyhedral group} is a non-cyclic finite subgroup of $\H^\times$ where $\H$ is the real quaternions. They are the generalised quaternion groups 
\[ Q_{4n} = \langle x, y \mid x^n=y^2, yxy^{-1}=x^{-1} \rangle\]
for $n \ge 2$ and the binary tetrahedral, octahedral and icosahedral groups $\widetilde{T}$, $\widetilde{O}$, $\widetilde{I}$ which are the preimages of the dihedral groups $D_{2n}$ and the symmetry groups $T$, $O$, $I$ under the double cover of Lie groups $f: \H^\times \hspace{-1mm} \cong S^3 \to SO(3)$.

We say that a group $G$ has \textit{$k$-periodic cohomology} for some $k \ge 1$ if its Tate cohomology groups satisfy $\hat{H}^i(G;\Z) \cong \hat{H}^{i+k}(G;\Z)$ for all $i \in \Z$.
Note that, if this is the case, then the periodicity must be induced by cupping with a class $g \in \wh H^k(G;\Z)$ (this follows from \cite[XII.11.1]{CE56}).    
It is easy to show that the binary polyhedral groups have $4$-periodic cohomology. The following can be found in \cite[XII.11.6]{CE56}.

\begin{prop} \label{prop:G-periodic}
If $G$ is a finite group, then the following are equivalent:
\begin{enumerate}[\normalfont (i)]
\item $G$ has periodic cohomology
\item $G$ has no subgroup of the form $C_p \times C_p$ for $p$ prime
\item The Sylow subgroups of $G$ are cyclic or generalised quaternionic $Q_{2^n}$.
\end{enumerate}
\end{prop}
 
Let $\SL_2(\F_p)$ denote the special linear group of degree $2$ over $\F_p$. Let $\TL_2(\F_p)$ denote the unique extension with kernel $\SL_2(\F_p)$ and quotient $C_2$ which has periodic cohomology. An explicit description of $\TL_2(\F_p)$ can be found in \cite[p384]{Wa10} and the proof that it is characterised by the given properties can be found in \cite[Proposition 2.2 (iii)]{Wa10}. Recall that we have
\[ \widetilde{T} \cong \SL_2(\F_3), \quad \widetilde{O} \cong \TL_2(\F_3), \quad \widetilde{I} \cong \SL_2(\F_5).\] 
For $G$ a finite group, let $O(G)$ be the unique maximal normal subgroup of odd order. 
 If $G$ has periodic cohomology,  then the \textit{type} of $G$ is determined by $G/O(G)$ as follows \cite[Corollary 3.6]{Wa10}. For reasons that will become apparent later, we will split II and V into two classes.

\vspace{-1mm}
\begin{figure}[h]
\begin{tabular}{|c|c|c|c|c|c|c|c|c|} 
\hline
 Type & I & IIa & IIb & III & IV & Va & Vb & VI \\ \hline
$G/O(G)$ & $C_{2^n}$ & $Q_8$ &
$Q_{2^n}$, $n \ge 4$
& $\widetilde{T}$ & $\widetilde{O}$ & $\widetilde{I}$ &
$\SL_2(\F_p)$, $p \ge 7$
 & 
$\TL_2(\F_p)$, $p \ge 5$
\\
\hline
\end{tabular}
\end{figure}
\vspace{-1mm}
\FloatBarrier

For the rest of this section, we will assume all groups are finite and will write $f: G \twoheadrightarrow H$ to denote a surjective group homomorphism. We will also assume basic facts about quaternion groups; for example, $Q_{2^n}$ has proper quotients $C_2$ and the dihedral groups $D_{2^m}$ for $1 < m <n$.
We begin with the following observation.

\begin{prop} \label{prop:type-closed}
Let $f: G \twoheadrightarrow H$ where $G$ and $H$ have periodic cohomology. If $|H| > 2$, then $G$ and $H$ have the same type.
\end{prop}

\begin{proof}
Note that $f(O(G)) \le H$ has odd order and so is contained in $O(H)$. In particular, $f$ induces a quotient $f: G/O(G) \twoheadrightarrow H/O(H)$. Hence it suffices to show that there are no (proper) quotients among groups in the family
\[ \mathscr{F} = \{C_{2^n}, Q_{2^m},\SL_2(\F_p), \TL_2(\F_p) : n \ge 2, m \ge 3, p \ge 3 \text{ prime}\}\] 
unless both are cyclic. Firstly, the quotients of $Q_{2^n}$ are $D_{2^m}$ for $1 < m < n$ and $C_2$ which are not in $\mathscr{F}$. It is easy to verify that the quotients of $\SL_2(\F_3)$ are $C_3$, $A_4$ and the quotients of $\TL_2(\F_3)$ are $C_2$, $S_3$, $S_4$, none of which are in $\mathscr{F}$.

For $p \ge 5$, $\SL_2(\F_p)$ has one (proper) normal subgroup $C_2$ with quotient $\PSL_2(\F_p)$ and similarly $\TL_2(\F_p)$ has normal subgroups $C_2$, $\SL_2(\F_p)$ with quotients $\text{PGL}_2(\F_p)$, $C_2$. To see this, note that they each surject onto $\PSL_2(\F_p)$ and $\text{PGL}_2(\F_p)$ respectively with kernel $C_2$ \cite[p384]{Wa10}. That $\PSL_2(\F_p)$ is a simple group then implies that any normal subgroups are of the required form.
That these groups are not in $\mathscr{F}$ follows, for example, from \cite[Proposition 2.3]{Wa10}.	
\end{proof}

We will split the remainder of this section into three parts. Firstly, we will determine the binary polyhedral quotients of groups $G$ with periodic cohomology. We will then use this to determine $m_{\H}(G)$ and finally we compare this with the Swan finiteness obstruction $\sigma_k(G)$.

\subsection{Binary polyhedral quotients}

If $G$ is a finite group, we say that two quotients $f_1 : G \twoheadrightarrow H_1$, $f_2 : G \twoheadrightarrow H_2$ are equivalent, written $f_1 \equiv f_2$, if $\Ker(f_1)=\Ker(f_2)$ are equal as sets (and so $H_1 \cong H_2$).

For a prime $p$, let $G_p$ be the isomorphism class of the Sylow $p$-subgroup of $G$. It is useful to note that, if $1 \to N \to G \to H \to 1$ is an extension, then there is an extension of abstract groups $1 \to N_p \to G_p \to H_p \to 1$.

\begin{lemma} \label{lemma:characterstic-subgroup-lemma}
Let $f: G \twoheadrightarrow H$ where $G$ and $H$ have periodic cohomology and $4 \mid |H|$. If $f' : G \twoheadrightarrow H'$ and $|H|=|H'|$, then $f \equiv f'$, i.e. $H \cong H'$ and $\Ker(f) = \Ker(f')$.
\end{lemma}

\begin{proof}
Let $H = G/N$, $H' = G/N'$ and define $\bar{G} = G/(N \cap N')$. Since there are successive quotients $G \twoheadrightarrow \bar{G} \twoheadrightarrow H$, we have  $G_p \twoheadrightarrow \bar{G}_p \twoheadrightarrow H_p$ for all primes $p$. 
If $G_p$ is cyclic, then this implies $\bar{G}_p$ is cyclic. If not, then $p=2$ and $G_2 = Q_{2^n}$ which has proper quotients $D_{2^m}$ for $2 \le m \le n-1$ and $C_2$. Since $H$ has periodic cohomology, $H_2$ is cyclic or generalised quaternionic and so $H_2 = Q_{2^n}$ since $4 \mid |H_2|$. Hence $\bar{G}_2 = Q_{2^n}$ since $G_2 \twoheadrightarrow H_2$ factors through $\bar{G}_2$, and so $\bar{G}$ has periodic cohomology.

Now note that $K=N/(N \cap N')$ and $K'=N'/(N \cap N')$ are disjoint normal subgroups of $\bar{G}$ and so $K \cdot K' = K \times K' \le \bar{G}$ by the recognition criteria for direct products. Hence $K \times K' \le \bar{G}$ and, since $\bar{G}$ has periodic cohomology, Proposition \ref{prop:G-periodic} (ii) implies that $|K|$ and $|K'|$ are coprime. Since $|N|=|N'|$, this implies that $|K|=|K'|=1$ and so $|N|=|N \cap N'| = |N'|$ and $N = N'$ as required.
\end{proof}

Now let $\mathcal{B}(G)$ denote the set of equivalence classes of quotients $f: G \twoheadrightarrow H$ where $H$ is a binary polyhedral group.
Since $4 \mid |H|$, applying Lemma \ref{lemma:characterstic-subgroup-lemma} again gives:

\begin{corollary} \label{cor:BPG-quotient-unique}
Let $G$ have periodic cohomology and let $f_1 , f_2 \in \mathcal{B}(G)$. Then $f_1 \equiv f_2 $ if and only if $\IM(f_1) \cong \IM(f_2)$.
\end{corollary}

In particular, this shows that $\mathcal{B}(G)$ is in one-to-one correspondence with the isomorphism classes of binary polyhedral groups $H$ which are quotients of $G$. For brevity, we will often write $H \in \mathcal{B}(G)$ when there exists $f: G \twoheadrightarrow H$ with $f \in \mathcal{B}(G)$.
 
 In order to determine $\mathcal{B}(G)$, it suffices to determine the set of maximal binary polyhedral quotients $\bm(G)$, i.e. the subset containing those $f \in \mathcal{B}(G)$ such that $f$ does not factor through any other $g \in \mathcal{B}(G)$. The rest of this section will be devoted to proving the following:

\begin{thm} \label{thm:bm-table}
If $G$ has periodic cohomology, then the type and the number of maximal binary polyhedral quotients $\# \bm(G)$ are related as follows.
\vspace{-1mm}
\begin{figure}[h]
\normalfont
\begin{tabular}{|c|c|c|c|c|c|c|c|c|} 
\hline 
 Type & I & IIa & IIb  & III & IV & Va & Vb & VI \\ \hline
$\# \bm(G) $ & 0,1 & 1 & 1,2,3 & 1 & 1 & 1 & 0 & 0
\\
\hline
\end{tabular}
\end{figure}
\vspace{-1mm}
\end{thm}

The entries in the table above denote the possible values that $\# \bm(G)$ can attain for a group $G$ of each type. When multiple values are listed, all values listed are attained by some group of the respective type.

\subsubsection*{Type I} Recall that $G$ has type I if and only if its Sylow subgroups are cyclic, and $G$ has a presentation
\[ G = \langle u,v \mid u^m=v^n=1, vuv^{-1}=u^r\rangle\] 
for some $r \in \Z/m$ where $r^n \equiv 1$ mod $m$ \cite[Lemma 4.1]{Wa10} and $(n,m)=1$. We will write $C_m \rtimes_{(r)} C_n$ to denote this presentation, where $C_n = \langle u \rangle$ and $C_m = \langle v \rangle$.
By \cite[p165]{Jo03a}, we can assume that $m$ is odd. 

If $G$ has a binary polyhedral quotient $H$, then Proposition \ref{prop:type-closed} implies that $H=Q_{4a}$ for $a>1$ odd and $4 \mid n$ since $m$ is odd.

\begin{lemma} \label{lemma:type-I}
Let $G = C_{m} \rtimes_{(r)} C_{4n}$. Then $G$ has a quotient $Q_{4a}$ if and only if $a \mid m$ and $r \equiv -1 \mod a$.
\end{lemma}

\begin{proof}
Recall that $Q_{4a} = C_a \rtimes_{(-1)} C_4$. If $a \mid m$ and and $r \equiv -1$ mod $a$, then $\langle u^a, v^4 \rangle \le G$ is normal since $r^4 \equiv 1$ mod $a$ implies $u v^4 u^{-1} = u^{1-r^4} v^4 \in \langle u^a, v^4 \rangle$. This implies that $G/\langle u^a,v^4 \rangle \cong C_a \rtimes_{(r)} C_4 = Q_{4a}$ since $r \equiv -1$ mod $a$.

Conversely, if $f : G \twoheadrightarrow Q_{4a}$, then $Q_{4a} \cong \langle f(u) \rangle \rtimes_{(r)} \langle f(v) \rangle$ and $|\langle f(u) \rangle| \mid m$, $|\langle f(u) \rangle| \mid 4n$. Since $Q_{4a}$ contains a maximal normal cyclic subgroup $C_{2a}$, and $m$ is odd, we must have $\langle f(u) \rangle \le C_a$. So $a \mid m$, which implies that $(a,4n)=1$ and $\langle f(u) \rangle \le C_4$ for some $C_4 \le Q_{4a}$. Hence $\langle f(u) \rangle = C_a$ and $\langle f(v) \rangle = C_4$ since they generate $Q_{4a}$. As $C_a \le Q_{4a}$ is unique and $C_4 \le Q_{4a}$ is unique up to conjugation, we have that $Q_{4a} = C_a \rtimes_{(-1)} C_4$ for any two subgroups isomorphic to $C_a$ and $C_4$. Since $\langle f(u) \rangle$ and $\langle f(v) \rangle$ are two such subgroups, this implies that $Q_{4a} \cong \langle f(u) \rangle \rtimes_{(-1)} \langle f(v) \rangle$ and so $r \equiv -1$ mod $a$.	
\end{proof}

Now suppose $G$ has two maximal binary polyhedral quotients $f_a: G \twoheadrightarrow Q_{4a}$, $f_b : G \twoheadrightarrow Q_{4b}$ for some $a,b >1$ odd, and we can assume $a$ is maximal. Then Lemma \ref{lemma:type-I} implies that $a,b \mid m$ and $r \equiv -1$ mod $a$ and $r \equiv -1$ mod $b$. If $d = \text{lcm}(a,b)$, then $d \mid m$ and $r \equiv -1$ mod $d$ and so there is a quotient $f_d : G \twoheadrightarrow Q_{4d}$ by Lemma \ref{lemma:type-I}.
By Corollary \ref{cor:BPG-quotient-unique} (or the proof of Lemma \ref{lemma:type-I}), $f_a$ and $f_b$ factor through $f_d$ which implies that $a=b=d$ as $f_a$ and $f_b$ are maximal. By Corollary \ref{cor:BPG-quotient-unique} again, this implies that $f_a$ and $f_b$ are equivalent. In particular, this shows that $\# \bm(G) \le 1$.

\subsubsection*{Type II} Recall that, if $G$ has type II, then $O(G) \le G$ has cyclic Sylow subgroups and so there exists $n \ge 3$ and $t$, $s$ odd coprime such that
\[ G \cong (C_t \rtimes_{(r)} C_s) \rtimes_{(a,b)} Q_{2^n}.\]
Furthermore, if $C_t = \langle u \rangle$, $C_s = \langle v \rangle$ and $Q_{2^n}$ is as above, then $Q_{2^n}$ acts via
\[ \varphi_x : u \mapsto u^a, v \mapsto v, \quad \varphi_y : u \mapsto u^b, v \mapsto v\]
for some $a, b \in \Z /t$ with $a^2 \equiv b^2 \equiv 1$ mod $t$ \cite[Theorem 4.6]{Wa10}. In what follows we will write $G = C_t \rtimes_{(a,b)} Q_{2^n}$ when $s=1$.

If $G$ has a binary polyhedral quotient $H$, then the proof of Proposition \ref{prop:type-closed} implies that $H/O(H)=Q_{2^n}$ and so $H=Q_{2^n m}$ for some $m$ odd.

\begin{lemma} \label{lemma:type-II-a}
Let $G = (C_t \rtimes_{(r)} C_s) \rtimes_{(a,b)} Q_{2^n}$. Then $G$ has a quotient $Q_{2^n m}$ if and only if $m \mid t$, $r \equiv 1$ mod $m$ and $Q_{2^n m} \cong C_m \rtimes_{(a,b)} Q_{2^n}$.
\end{lemma}

\begin{proof}
If $m \mid t$ and $r \equiv 1$ mod $m$, then $\langle u^m,v \rangle \le G$ is normal since $uvu^{-1}=u^{1-r}v \in \langle u^m,v \rangle$. This implies that $G/\langle u^m,v \rangle \cong C_t \rtimes_{(a,b)} Q_{2^n}$ which has quotient $C_m \rtimes_{(a,b)} Q_{2^n}$ since $m \mid t$. If $Q_{2^n m} \cong C_m \rtimes_{(a,b)} Q_{2^n}$, then $G$ has quotient $Q_{2^n m}$.

Conversely, suppose $f : G \twoheadrightarrow Q_{2^n m}$. Let $h : G \twoheadrightarrow G/\langle u, v \rangle \cong Q_{2^n}$  and note that, if $g: Q_{2^n m} \twoheadrightarrow Q_{2^n}$, then $\Ker(g \circ f) = \Ker(h) = \langle u,v \rangle$ by Corollary \ref{cor:BPG-quotient-unique} and so $\Ker(f) \le \langle u,v \rangle$. 
By composing $g$ with an element of $\Aut(Q_{2^n})$, we can assume $g \circ f = h$ and so $Q_{2^n m} \cong \Ker(g) \rtimes \langle f(x), f(y) \rangle$. Since $f(v) \in \Ker(g)$ has a trivial action by $\langle f(x), f(y) \rangle \cong Q_{2^n}$, this implies $f(v)=1$, i.e. $v \in \Ker(f)$. This implies $\Ker(f) = \langle u^{\ell},v\rangle$ for some $\ell \mid t$ and we need $\ell = m$ since $\Ker(f) \le G$ has index $2^n m$. Hence $m \mid t$ and, by normality, $uvu^{-1}=u^{1-r}v \in \langle u^m,v\rangle$ and so $r \equiv 1$ mod $m$. Finally, we have $Q_{2^n m} \cong G/\langle u^m,v\rangle \cong C_m \rtimes_{(a,b)} Q_{2^n}$.
\end{proof}

\begin{lemma} \label{lemma:type-II-b}
If $m \ge 1$, then $Q_{2^n m } \cong C_m \rtimes_{(a,b)} Q_{2^n}$  if and only if
\[ (a,b) = \begin{cases} (1,-1),(-1,1),(-1,-1), & \text{if $n=3$}\\ (1,-1), & \text{if $n \ge 4$.}\end{cases}\]
\end{lemma}

\begin{proof}
It follows easily from the standard presentation that $Q_{2^n m} \cong C_m \rtimes_{(1,-1)} Q_{2^n}$. If $f: Q_{2^n m} \twoheadrightarrow Q_{2^n}$, then $\Ker(f) = C_m$ is unique by Corollary \ref{cor:BPG-quotient-unique}. Hence $Q_{2^n m } \cong C_m \rtimes_{(a,b)} Q_{2^n}$ if and only if there exists $\theta \in \Aut(Q_{2^n})$ such that $\varphi_{(a,b)} = \varphi_{(1,-1)} \circ \theta$ where $\varphi_{(i,j)} : Q_{2^n} \to \Aut(C_m) \cong (\Z/m)^\times$ has $\varphi_{(i,j)}(x)=i$, $\varphi_{(i,j)}(y)=j$.
This implies that $\IM(\varphi_{(a,b)}) \le \IM(\varphi_{(1,-1)}) = \langle 1,-1 \rangle =\{ \pm 1\}$ and so $a,b \in \{ \pm 1\}$. If $(a,b)= (1,1)$, then $Q_{2^n m} \cong C_m \times Q_{2^n}$ which is a contradiction unless $m=1$, in which case $(1,1)=(1,-1)$. In particular, $(a,b) \in \{ (1,-1),(-1,1),(-1,-1)\}$.

If $n=3$, then we $\theta_1 : x \mapsto y, y \mapsto x$ has $\varphi_{(1,-1)} \circ \theta_1 = \varphi_{(1,-1)}$ and $\theta_2 : x \mapsto y, y \mapsto xy$ has $\varphi_{(1,-1)} \circ \theta_2 = \varphi_{(-1,-1)}$. Hence all $(a,b)$ are possible. If $n \ge 4$, then
\[\Aut(Q_{2^n}) =\{ \theta_{i,j} : x \mapsto x^i, y \mapsto x^{j}y \mid i \in (\Z/{2^{n-1}})^\times, j \in \Z/{2^{n-1}}\}\]
and $\varphi_{(1,-1)} \circ \theta_{i,j} = \varphi_{(1,-1)}$ for all $i, j$ and so only $(a,b) = (1,-1)$ is possible.
\end{proof}

Now suppose $G$ has type IIb, i.e. $G/O(G) = Q_{2^n}$ for some $n \ge 4$.  By combining Lemmas \ref{lemma:type-II-a} and \ref{lemma:type-II-b}, we get that $G$ has a quotient $Q_{2^n m}$ if and only if $m \mid t$, $r \equiv 1$ mod $m$ and $(a,b)=(1,-1)$ mod $m$. If $G$ has two distinct maximal binary polyhedral quotients $f_i : G \twoheadrightarrow Q_{2^n m_i}$ for $i=1,2$, then $m_1, m_2 \mid t$, $r \equiv 1$ mod $m_1, m_2$ and $(a,b) \equiv (1,-1)$ mod $m_1, m_2$. If $m = \text{lcm}(m_1,m_2)$, then this implies that $m \mid t$, $r \equiv 1$ mod $m$ and $(a,b) \equiv (1,-1)$ mod $m$ and so $f: G \twoheadrightarrow Q_{2^n m}$. By Corollary \ref{cor:BPG-quotient-unique}, $m > m_1,m_2$ and $f_1$ and $f_2$ must factor through $f$ which is a contradiction. Hence $\#\bm(G) =1$.

A similar argument works in the case where $G$ has type IIa, i.e. $G/O(G) = Q_8$. If $G$ has four distinct maximal binary polyhedral quotients $f_i : G \twoheadrightarrow Q_{8 m_i}$ for $i=1,2,3,4$, then Lemmas \ref{lemma:type-II-a} and \ref{lemma:type-II-b} imply there exists $i, j$ for which $(a,b) \equiv (1,-1)$, $(-1,1)$ and $(-1,-1)$ mod $m_i,m_j$. By a similar argument to the above, this implies that $f_i$, $f_j$ factors through $f: G \twoheadrightarrow Q_{8 m}$ where $m=\text{lcm}(m_i,m_j)$ which is a contradiction since $m_i \ne m_j$ and $f_i$, $f_j$ are maximal. Hence $1 \le \#\bm(G) \le 3$. 

Furthermore, if $G$ has quotients $Q_{8 m_i}$ and $Q_{8 m_j}$, then this implies that $(a,b)$ mod $m_i$ and $(a,b)$ mod $m_j$ are distinct which is a contradiction unless $(m_i,m_j)=1$.

\subsubsection*{Types III, IV, Va}

If $G$ has type III, IV or Va, then $G/O(G) = \widetilde{T}$, $\widetilde{O}$ or $\widetilde{I}$. If $f:G \twoheadrightarrow H$ is a binary polyhedral quotient, then $H \cong G/O(G)$ by Proposition \ref{prop:type-closed}. By Corollary \ref{cor:BPG-quotient-unique}, $f$ is equivalent to $G \twoheadrightarrow G/O(G)$. Hence $\# \bm(G), \# \mathcal{B}(G) =1$.

\subsubsection*{Types Vb, VI} Suppose $G$ has type Vb or VI. Since no binary polyhedral groups have type Vb or VI, Proposition \ref{prop:type-closed} implies that $G$ has no binary polyhedral quotients. Hence $\# \bm(G), \# \mathcal{B}(G) =0$. This completes the proof of Theorem \ref{thm:bm-table}.

\subsection{Quaternionic representations}

Recall that $m_{\H}(G)$ denotes the number of copies of $\H$ in the Wedderburn decomposition of $\R G$ for a finite group $G$. This coincides with the number of irreducible quaternionic representations of $G$ of (quaternionic) dimension one. We say that a finite group $G$ is said to satisfy the \textit{Eichler condition} if $m_{\H}(G)=0$. The following is  \cite[Proposition 3.3]{Ni18}.

\begin{prop} \label{prop:relative-eichler-group}
Let $f: G \twoheadrightarrow H$ be a quotient. Then $m_{\H}(G) = m_{\H}(H)$ if and only if every $g \in \mathcal{B}(G)$ factors through $f$, i.e. if $f^* : \mathcal{B}(H) \to \mathcal{B}(G)$ is a bijection.
\end{prop}

For example, this shows that $G$ satisfies the Eichler condition if and only if $G$ has no quotient which is a binary polyhedral group. It also follows that, if $G$ has a unique maximal binary polyhedral quotient $H$, then $m_{\H}(G)=m_{\H}(H)$. 

We now show how to use this to deduce the following from Theorem \ref{thm:bm-table}.

\begin{thm} \label{thm:mh-table}
If $G$ has periodic cohomology, then type and $m_{\H}(G)$ are related as follows.
\vspace{-1mm}
\begin{figure}[h]
\normalfont
\begin{tabular}{|c|c|c|c|c|c|c|c|c|}
\hline 
 Type & I & IIa & IIb  & III & IV & Va & Vb & VI \\ \hline
$m_{\H}(G) $ & $\ge 0$ & $\ge 1$ odd & $\ge 2$ even & 1 & 2 & 2 & 0 & 0
\\
\hline
\end{tabular}
\end{figure}
\vspace{-1mm}
\end{thm}

\subsubsection*{Type Vb, VI}

If $G$ has type Vb or VI, then Theorem \ref{thm:bm-table} implies that $G$ has no binary polyhedral quotients and so $m_{\H}(G)=0$ by Proposition \ref{prop:relative-eichler-group}.

\subsubsection*{Type IIb, III, IV, Va}

If $G$ has type IIb, III, IV or Va, then Theorem \ref{thm:bm-table} implies that $\#\bm(G)=1$, i.e. $G$ has a unique maximal binary polyhedral quotient $H$. By Proposition \ref{prop:relative-eichler-group}, we must have that $m_{\H}(G)=m_{\H}(H)$. By Proposition \ref{prop:type-closed}, $H$ has the same type as $G$. Recall that $m_{\H}(Q_{4n}) = \lfloor n/2 \rfloor$ \cite[Section 12]{Jo03a}. If $G$ has type IIa, then $H = Q_{2^n m}$ for $n \ge 4$, $m \ge 1$ odd and $m_{\H}(Q_{2^n m}) = 2^{n-3} m \ge 2$ is even. If $G$ has type III, IV or Va, then $H = \widetilde{T}$, $\widetilde{O}$ or $\widetilde{I}$ respectively which have $m_{\H}(\widetilde{T})=1$, $m_{\H}(\widetilde{O})=2$ and $m_{\H}(\widetilde{I})=2$.

\subsubsection*{Type IIa}

If $G$ has type IIa, then Theorem \ref{thm:bm-table} implies that $\#\bm(G)=1,2$ or $3$. If $b = \#\bm(G)$, let $f_i: G \twoheadrightarrow Q_{8 m_i}$ denote the maximal binary polyhedral quotients for $1 \le i \le b$. It follows from the proof of Theorem \ref{thm:bm-table} that the $m_i$ are coprime and so the maximal quotient factoring through any two of the $f_i$ is the unique quotient $f: G \twoheadrightarrow Q_8$. Since $m_{\H}(Q_{8 m_i}) = m_i$ and $m_{\H}(Q_8)=1$, it can be shown using real representation theory that
\[ m_{\H}(G) = \sum_{i=1}^b (m_{\H}(Q_{8 m_i})-1) + m_{\H}(Q_8) =
\begin{cases}
m_1, & \text{if $b=1$}\\
(m_1+m_2)-1, & \text{if $b=2$}\\
(m_1+m_2+m_3)-2, & \text{if $b=3$}
\end{cases}
\]
which is odd since the $m_i$ are odd. This completes the proof of Theorem \ref{thm:mh-table}.

\subsection{Vanishing of the Swan finiteness obstruction}

Recall that a group $G$ has $k$-periodic cohomology if and only if there exists a $k$-periodic projective resolution over $\Z G$ (see, for example, \cite[Proposition 3.9]{Ni20a}). If $G$ has $k$-periodic cohomology, then Swan \cite{Sw60a} defined an obstruction $\sigma_k(G) \in C(\Z G)/T_G$, where $T_G$ is generated by $(I,r)$ for $r \in (\Z/|G|)^\times$, which vanishes if and only if there exists a $k$-periodic resolution of free $\Z G$-modules. Determining which groups have $\sigma_k(G) = 0$ remains a difficult open problem, and has applications to the classification of spherical space forms \cite{DM85}.

The main result of this section will be the following extension of Theorem \ref{thm:mh-table} which shows how $m_{\H}(G)$ and the vanishing of $\sigma_k(G)$ are related to the type of $G$.

\begin{thm} \label{thm:big-table}
The columns of the following table list the triples of type, $m_{\H}(G)$ and $\sigma_k(G)$ which occur for groups $G$ with periodic cohomology.
\vspace{-1mm}
\begin{figure}[h]
\normalfont
\begin{tabular}{|c|c|c|c|c|c|c|c|c|c|}
 \hline 
Type & I & IIa & IIb & III & IV & V a & V b & VI \\
 \hline
  $m_{\H}(G)$ & $\ge 0$ & $\ge 1$ odd & $\ge 2$ even & $1$ & $2$ & $2$ & $0$ & $0$ \\
 \hline
 $\sigma_k(G)$ & $0$ & $0$ or $\ne 0$ & $0$ or $\ne 0$ & $0$ & $0$ or $\ne 0$ & $0$ & $0$ & $0$ or $\ne 0$ \\
  \hline
\end{tabular}
\end{figure}
\vspace{-1mm}
\end{thm}

In order to prove Theorem \ref{thm:big-table}, we will begin by noting the following which is proven in \cite[Theorem 3.19]{DM85}.

\begin{lemma}
If $G$ has $k$-periodic cohomology and type I, III or V, then $\sigma_k(G)=0$.	
\end{lemma}

This gives the restrictions on the vanishing of $\sigma_k(G)$ given in the above table. It now suffices to construct examples which realise the constraints in each column.

We begin by constructing the examples with $\sigma_k(G)=0$. Firstly the groups $C_{n}$, $Q_{4n}$ and $\SL_2(\F_p)$, $\TL_2(\F_p)$ for $p \ge 3$ can all be shown to have vanishing finiteness obstruction \cite[Theorem 3.19 (c)]{DM85}.
Now note that $m_{\H}(Q_{4n}) = \lfloor n/2 \rfloor$ \cite[Section 12]{Jo03a}. In particular, $m_{\H}(Q_{8n+4}) = m_{\H}(Q_{8n}) = n$ for all $n \ge 1$ and $Q_{8n+4}$ has type I and $Q_{8n}$ has type II.
Finally, the cyclic groups $G=C_n$ have type I and $m_{\H}(C_n)=0$. 

If $n \ge 3$ and $a,b,c \ge 1$ are odd coprime, then define
\[ Q(2^n a;b,c) = C_{abc} \rtimes_{(r,s)} Q_{2^n}\]
where $(r,s) \equiv (1,-1)$ mod $a$, $(r,s) \equiv (-1,-1)$ mod $b$ and $(r,s) \equiv (-1,1)$ mod $c$.
The following was shown by Milgram \cite[Theorem D]{Mi85} and Davis \cite[Corollary 6.2]{Da81}.

\begin{thm}\label{thm:milg}
Let $p$, $q$ be distinct odd primes. 
\begin{enumerate}[\normalfont (i)]
\item Then $\sigma_4(Q(8;p,q)) \ne 0$ if $p, q \equiv 3$ mod $4$, or $p \equiv 3$ mod $4$, $q \equiv 5$ mod $8$ and $p^n \not \equiv \pm 1$ mod $q$ for all $n$ odd
\item If $n \ge 4$, then $\sigma_4(Q(2^n;p,1)) \ne 0$ if $p \not \equiv 1$ mod $8$ and $p \not \equiv \pm 1$ mod $2^{n-1}$.
\end{enumerate}
\end{thm}

We will construct further examples of groups with $\sigma_k(G) \ne 0$ as follows. Let $H$ be a group listed in Theorem \ref{thm:milg} which has $\sigma_4(H) \ne 0$ and suppose $H \le G$ where $G$ has $4n$-periodic cohomology for some $n \ge 1$ odd. If $\sigma_{4n}(G) = 0$, then $\sigma_{4n}(H) = 0$ by restricting resolutions. However, since $4n \equiv 4$ mod 8, it follows from \cite[Corollary 12.6]{Wa79a} that $\sigma_4(H) = 0$. This is a contradiction and so $\sigma_{4n}(G) \ne 0$.

\subsubsection*{Type IIa} 

It suffices to prove the following.

\begin{lemma}
Let $n \ge 3$, $m \ge 1$ be odd coprime. Then there exists distinct odd primes $p$, $q$ and $r, a, b \in \Z$ such that
\[ G = (C_{mpq} \rtimes_{(r)} C_n) \rtimes_{(a,b)} Q_8\]
has $\sigma_{4n}(G) \ne 0$ and $m_{\H}(G) = m$.
\end{lemma}

\begin{proof}
By Dirichlet's theorem on primes in arithmetic progression, there exists distinct primes $p$, $q$ such that $p,q \nmid m$, $p, q \equiv 3$ mod 4 and $p, q \equiv 1$ mod $n$. Let $a,b \in \Z/mpq$ be such that $(a,b) \equiv (1, -1)$ mod $m$, $(a,b) \equiv (-1,1)$ mod $p$ and $(a,b) \equiv (1,-1)$ mod $q$. Since $n \mid p-1,q-1$, we can pick $r_p \in (\Z/p)^\times$, $r_q \in (\Z/q)^\times$ of order $n$. Let $r \in \Z/mpq$ be such that $r \equiv 1$ mod $m$, $r \equiv r_p$ mod $p$, $r \equiv r_q$ mod $q$. This has $r^n \equiv 1$ mod $mpq$ and so we can define $G = (C_{mpq} \rtimes_{(r)} C_n) \rtimes_{(a,b)} Q_8$.

By \cite[Corollary 5.6]{Wa10}, $G$ has period $4 \cdot \ord_{mpq}(r) = 4n$. If $C_{mpq} = \langle u \rangle$, $C_n = \langle v \rangle$ and $Q_8 = \langle x,y \rangle$, then $Q(8;p,q) \cong \langle u^m, x, y \rangle \le G$. Since $p,q \equiv 3$ mod 4, Theorem \ref{thm:milg} (i) implies that $\sigma_4(Q(8;p,q)) \ne 0$ and so $\sigma_{4n}(G) \ne 0$.

Since $(a,b) \equiv (1,-1)$ mod $m$ and $r \equiv 1$ mod $m$, Lemmas \ref{lemma:type-II-a} and \ref{lemma:type-II-b} imply that $G$ has a quotient $Q_{8m}$. Similarly, if $G$ has a quotient $Q_{8m_0}$ with $(m_0, m)=1$, then $m_0 \mid mpq$ and $r \equiv 1$ mod $m_0$. However, $r \nequiv 1$ mod $p,q$ and so $m_0=1$. This implies that $\bm(G) = \{ Q_{8m}\}$ and so $m_{\H}(G) = m_{\H}(Q_{8m})=m$ by Proposition \ref{prop:relative-eichler-group}.	
\end{proof}

\subsubsection*{Type IIb}

Let $n \ge 1$. By Dirichlet's theorem on primes in arithmetic progression, there exists a prime $p$ such that $p \nmid n$, $p \not \equiv 1$ mod $8$ and $p \not \equiv \pm 1$ mod $2^{k+3}$ where $k= \nu_2(n)$ is the highest power of $2$ dividing $n$. If $G = Q(16n;p,1)$, then Theorem \ref{thm:milg} (ii) implies that $\sigma_4(G) \ne 0$. It is easy to see that $\bm(G) = \{ Q_{16n}\}$ and so $m_{\H}(G) = m_{\H}(Q_{8n})=n$ by Proposition \ref{prop:relative-eichler-group}.	

\subsubsection*{Type IV}

In order to give our example, first define $\widetilde{T}_v$ for $v \ge 1$ as the unique extension with kernel $C_{3^{v-1}}$ and quotient $\widetilde{T}$ which has periodic cohomology \cite[Lemma 4.3]{Wa10}. Also define $\widetilde{O}_v$ for $v \ge 1$ to be the unique extension with kernel $\widetilde{T}_v$ and quotient $C_2$ which has periodic cohomology \cite[Lemma 4.4]{Wa10}. Note that $\widetilde{O}_v$ is $4$-periodic by \cite[Corollary 5.6]{Wa10} and, by Proposition \ref{prop:type-closed}, it also has type IV since it has a quotient $\widetilde{O}$.

 By \cite[Theorem 5.1]{Wa10}, $\widetilde{O}_v$ has a subgroup of the form $Q(16;3^{v-1},1)$ and so has a subgroup $Q(16;3,1)$ for all $v \ge 2$. 
Since $3 \nequiv \pm 1$ mod 8, Theorem \ref{thm:milg} (ii) implies that $\sigma_4(Q(16;3,1)) \ne 0$ and so $\sigma_4(\widetilde{O}_v) \ne 0$ for all $v \ge 2$.

\subsubsection*{Type VI}

Recall that $\TL_2(\F_p)$ is an extension with kernel $\SL_2(\F_p)$ and quotient $C_2$. Let $z \in \TL_2(\F_p)$ map to the generator of $C_2$. 
Let $p \ge 5$ and $q \nmid p(p^2-1)$ be odd primes and define
\[ G = C_q \rtimes_\varphi \TL_2(\F_p)\]
where, if $C_q =\langle u \rangle$, then $\varphi_z : u \mapsto u^{-1}$ and $\varphi_x : u \mapsto u$ for all $x \in \SL_2(\F_p) \le \TL_2(\F_p)$. Since $q$ is coprime to $|\TL_2(\F_p)|=2p(p^2-1)$, \cite[Theorem 4.6]{Wa10} implies $G$ has type VI. It follows from \cite[Theorem 5.2]{Wa10} that $G$ has subgroups $Q(4(p+1);q,1)$ and $Q(4(p-1);q,1)$ and so $G$ has a subgroup $Q(2^n;q,1)$ where $n = 2+\max \{\nu_2(p+1),\nu_2(p-1)\} \ge 4$. Note also that, by \cite[Corollary 5.6]{Wa10}, $G$ has period $2(p-1)$.

If $p \equiv 3$ mod $4$, then $2(p-1) \equiv 4$ mod 8. Using Dirichlet's theorem of primes in arithmetic progression, we can now pick a prime $q$ such that $q \nmid p(p^2-1)$, $q \nequiv 1$ mod 8 and $q \nequiv 1$ mod $2^{n-1}$. By Theorem \ref{thm:milg}, we have $\sigma_4(Q(2^n;q,1)) \ne 0$ and so $\sigma_{2(p-1)}(Q(2^n;q,1)) \ne 0$ since $2(p-1) \equiv 4$ mod 8. Since $Q(2^n;q,1) \le G$, this implies that $\sigma_{2(p-1)}(G) \ne 0$.

\section{Induced representations and the action of $\Aut(G)$}
\label{section:quotients}

If $G$ is a finite group and $P \in P(\Z G)$, we say that $\Aut(G)$ \textit{acts} on $[P]$ if there exists a group homomorphism $\psi: \Aut(G) \to (\Z/|G|)^\times$ for which $[P] - [P_\theta] = [(I,\psi(\theta))] \in C(\Z G)$ where $I$ is the augmentation ideal. The action is then given by sending $P_0 \mapsto (P_0)_\theta \otimes (I,\psi(\theta))$ for any $P_0 \in [P]$ of rank one.
For example, $\Aut(G)$ acts on the stably free class $[\Z G]$ for all finite groups $G$ via the trivial map $\psi(\theta) =1$ for all $\theta \in \Aut(G)$. More generally, if $G$ has periodic cohomology, then $\Aut(G)$ acts on $[P_{(G,n)}]$ by the action defined in the introduction (see also \cite[Section 7]{Ni20a}). 

Recall that, if $R$ and $S$ are rings and $f: R \to S$ is a ring homomorphism, then $S$ is an $(S,R)$-bimodule, with right-multiplication by $r \in R$ given by $x \cdot r = x f(r)$ for any $x \in S$. If $M$ is an $R$-module, we can define the \textit{extension of scalars} of $M$ by $f$ as the tensor product
\[ f_\#(M) = S \otimes_{R} M \]
which is defined since $S$ as a right $R$-module and $M$ as a left $R$-module. We will view this as a left $S$-module where left-multiplication by $s \in S$ is given by $s \cdot (x \otimes m) = (sx)\otimes m$ for any $x \in S$ and $m \in M$.

For example, a group homomorphism $f: G \to H$ can be viewed as a ring homomorphism $f : \Z G \to \Z H$ by sending $\sum_{g \in G} x_g g \mapsto \sum_{g \in G} x_g f(g)$ where $x_g \in \Z$. If $M$ is a $\Z G$-module then $f_\#(M) = \Z H \otimes_{\Z G} M$ coincides with the induced module.

Recall that a subgroup $N \le G$ is \textit{characteristic} if every $\varphi \in \Aut(G)$ has $\varphi(N)=N$ and we say that a quotient $f: G \to H$ is characteristic if $\Ker(f) \le G$ is a characteristic subgroup.

\begin{lemma} \label{lemma:characteristic-quotient-induced-action}
Let $f: G \twoheadrightarrow H$ be a characteristic quotient and let $\theta \in \Aut(G)$. If $M$ is a $\Z G$-module, then
\[ f_\#(M_{\theta}) \cong f_\#(M)_{\bar{\theta}}.\]
\end{lemma}

\begin{proof}
Since $f$ is characteristic, there is a map $\bar{\cdot} : \Aut(G) \to \Aut(H)$ such that $f \circ \theta = \bar{\theta} \circ f$ for all $\theta \in \Aut(G)$. The result now follows from \cite[Corollary 8.4]{Ni20a}.
\end{proof}

\begin{lemma} \label{lemma:swan-quotient}
Let $f: G \twoheadrightarrow H$ and $r \in (\Z/|G|)^\times$, then
\[ f_\#((I_G,r)) \cong (I_H,r).\]	
\end{lemma}

\begin{proof}
By \cite[Remark 2.3]{DM85}, we have that $(I_G,r) \cong (\Sigma_G,r^{-1})$ where $\Sigma_G = \sum_{g \in G} g$ denotes the group norm. Hence it suffices to prove instead that $f_\#((\Sigma_G,r)) \cong (\Sigma_H,r)$.
Let $d=|G|/|H|$ and, since $(d,r)=1$, there exists $a, b \in \Z$ such that $ad+br=1$. Now define
\[ \psi : (\Sigma_H,r) \to f_\#((\Sigma_G,r)) = \Z H \otimes_{\Z G} (\Sigma_G,r)\]
by sending $r \mapsto 1 \otimes r$ and $\Sigma_H \mapsto a(1 \otimes \Sigma_G)+b \Sigma_H(1 \otimes r)$, and it is straightforward to check that $\psi$ is a $\Z H$-module isomorphism.
\end{proof}

Our main result gives a possible way to show how $[P]/\Aut(G)$ has non-cancellation when $f: G \twoheadrightarrow H$ is a characteristic quotient and $f_* = \bar{\cdot} : \Aut(G) \to \Aut(H)$.

\begin{thm} \label{thm:P/Aut(G)}
Let $f: G \twoheadrightarrow H$ is a characteristic quotient and $P \in P(\Z G)$. If $\Aut(G)$ acts on $[P]$ and $\bar{P} = f_\#(P)$, then there is a surjection of graded trees
\[ f_\# : [P]/\Aut(G) \twoheadrightarrow [\bar{P}]/\IM(f_*)\]
where the action of $\IM(f_*) \subseteq \Aut(H)$ on $[\bar{P}]$ is induced by $f$.
\end{thm}

\begin{proof}
If $\otimes$ denotes the tensor product of $\Z G$-modules, then it is easy to see that  $f_\#(M \otimes N) \cong f_\#(M) \otimes N$ for all $M$, $N$. Combining with Lemmas  \ref{lemma:characteristic-quotient-induced-action} and \ref{lemma:swan-quotient} gives
\begin{align*} f_\#((I_G,r) \otimes P_\theta) &\cong (I_H,r) \otimes_{\Z H} P_\theta \cong ((I_H,r) \otimes \Z H) \otimes_{\Z H} P_\theta \\
& \cong (I_H,r) \otimes (\Z H \otimes_{\Z H} P_\theta) \cong (I_H,r) \otimes \bar{P}_{\bar{\theta}}
\end{align*}
where $\bar{\cdot} : \Aut(G) \to \Aut(H)$. Since $[(I_G,r)\otimes P_\theta] = [P]$, this implies that $[(I_H,r) \otimes \bar{P}_{\bar{\theta}}] = [\bar{P}]$ and so $\Aut(H)$ acts on $[\bar{P}]$ by sending $P_0$ to $(I_H,\varphi(\theta)) \otimes P_0$. 

Now note that the map $f_\# : [P] \to [\bar{P}]$ is well-defined and surjective by \cite[Theorem A10]{Sw83}. By the argument above, $f_\#$ respects these actions and so induces a map $f_\# : [P]/\Aut(G) \to [\bar{P}]/\Aut(H)$ which is necessarily surjective.	
\end{proof}

We will now show how this can be applied in the case where $G$ has periodic cohomology. First note the following:

\begin{lemma} \label{lemma:MH>=3}
If $G$ has periodic cohomology and $m_{\H}(G) \ge 3$, then $G$ has a quotient $Q_{4n}$ for some $n \ge \max\{\frac{2}{3} m_{\H}(G),6\}$.
\end{lemma}

\begin{proof}
Since $m_{\H}(G) \ge 3$, the possible types are I and II by Theorem \ref{thm:big-table}. If $H$ is the binary polyhedral quotient of maximal order, then Proposition \ref{prop:type-closed} implies $H=Q_{4n}$ for some $n \ge 2$. If $b = \#\bm(G)$, then Theorem \ref{thm:bm-table} implies $b=1,2,3$. 

Let $m=m_{\H}(G)$. If $b=1$, then Proposition \ref{prop:relative-eichler-group} implies $m = m_{\H}(Q_{4n}) = \lfloor \frac{n}{2} \rfloor \le \frac{n}{2}$ and so $n \ge 2 m$. 
If $b =2,3$, then $G$ has type IIa. If $Q_{8m_i}$ are the maximal binary polyhedral quotients for $1 \le i \le b$, then the proof of Theorem \ref{thm:bm-table} implies that $m = (m_1+m_2)-1$ if $b=2$ and $m = (m_1+m_2+m_3)-2$ if $b=3$. Suppose $m_1 > m_2 > m_3$ so that $m_1 \ge m_2+2$ and $m_2 \ge m_3+2$ since the $m_i$ are odd coprime. Since $n = 2m_1$, this implies that $m \le n-3$ if $b=2$ and $m \le \frac{3}{2} n - 8$ if $b=3$. Hence $n \ge m +3$ and $n \ge \frac{2}{3} (m+8)$ in the two cases respectively. The bound now follows since $\min\{2m,m+3,\frac{2}{3} (m+8)\} \ge \max\{6,\frac{2}{3} m\}$ for all $m \ge 3$.
\end{proof}

In fact, the quotient $f: G \twoheadrightarrow Q_{4n}$ is always characteristic due to the following.

\begin{prop} \label{prop:quotient=characteristic}
If $f: G \twoheadrightarrow H$ where $G$ has periodic cohomology and $H$ is a binary polyhedral group, then $f$ is characteristic.
\end{prop}

\begin{proof}
Let $\varphi \in \Aut(G)$ and consider $N = \varphi(\Ker(f)) \le G$. Then $N$ is a normal subgroup with $|N|=|\Ker(f)|$. Since $H$ is a binary polyhedral group, it has $4 \mid |H|$ and so Lemma \ref{lemma:characterstic-subgroup-lemma} implies that $N = \Ker(f)$.	
\end{proof}

By combining this with Theorem \ref{thm:P/Aut(G)}, we have:

\begin{corollary} \label{cor:P/Aut(G)-periodic}
Let $G$ have $k$-periodic cohomology and $m_{\H}(G) \ge 3$. Then there is a surjection of graded trees
\[ f_\# : [P_{(G,k)}]/\Aut(G) \twoheadrightarrow [\widebar{P_{(G,k)}}]/\IM(f_*)\]
where $f: G \twoheadrightarrow Q_{4n}$ for some $n \ge \max\{\frac{2}{3} m_{\H}(G),6\}$ and where the action of $\IM(f_*) \subseteq \Aut(Q_{4n})$ on $[\widebar{P_{(G,k)}}]$ is induced by $f$.
\end{corollary}

Note that the two bounds for $n$ have distinct uses. If $m_{\H}(G)$ is small, as is the case when dealing with cancellation in Theorem \ref{thm:main}, then the bound $n \ge 6$ will be most useful. If $m_{\H}(G)$ is large, as in the asymptotic estimates in Theorem \ref{thm:main-bounds}, then we will use the bound $n \ge \frac{2}{3} m_{\H}(G)$.

\section{Cancellation for the Swan finiteness obstruction} \label{section:cancellation-periodic}

The aim of this section will be to use the results in Section \ref{section:group-theory} to prove the following cancellation result. 

\begin{thm} \label{thm:main-cancellation}
Let $G$ have $k$-periodic cohomology, let $n=ik$ or $ik-2$ for some $i \ge 1$ and let $P_{(G,n)}$ be a projective $\Z G$-module for which $[P_{(G,n)}] = \sigma_k(G) \in C(\Z G)/ T_G$.
Then $[P_{(G,n)}]$ has cancellation if and only if  $m_{\H}(G) \le 2$.
\end{thm}

The forward implication, that $m_{\H}(G) \ge 3$ implies $[P_{(G,n)}]$ has non-cancellation, follows immediately by combining Lemma \ref{lemma:MH>=3}  with \cite[Theorem A]{Sw83} (see also Lemma \ref{lemma:q4n-quotient}). Note that it also follows from the stronger statement that $[P_{(G,n)}]/\Aut(G)$ has non-cancellation, which will be proven in Section \ref{section:non-cancellation}.
The remainder of this section will be devoted to the proof of the backward direction, that $m_{\H}(G) \le 2$ implies $[P_{(G,n)}]$ has cancellation.

Recall that, if $G$ is finite, $I = \Ker(\varepsilon: \Z G \to \Z)$ is the augmentation ideal and $(r,|G|)=1$, then $(I,r) \subseteq \Z G$ is a projective ideal. Since $(I,r) \cong (I,s)$ if $r \equiv s$ mod $|G|$, we often write $r \in (\Z/|G|)^\times$,
As noted in \cite[Section 4.2]{Ni20a}, if $P$ is a projective $\Z G$-module and $r \in (\Z/|G|)^\times$, then we can define $(I,r) \otimes P$ to be a (left) $\Z G$-module since $(I,r)$ is a two-sided ideal. We can use this action to show the following.

\begin{prop} \label{prop:cancellation-mod-T}
Let $G$ be a finite group let $P, Q$ be projective $\Z G$-modules such that $[P] - [Q] \in T_G$. Then there is an isomorphism of graded trees \[[P] \cong [Q]\] given by sending $P_0 \oplus \Z G^m \mapsto ((I,r) \otimes P_0) \oplus \Z G^m$ for some $r \in (\Z/|G|)^\times$ where $P_0$ is a rank one projective $\Z G$-module and $m \ge 0$ is an integer.
\end{prop}

\begin{proof}
Let $r \in (\Z/|G|)^\times$ be such that $[Q]=[P]+[(I,r)] \in C(\Z G)$. By \cite{Sw60b}, an arbitrary element of $[P]$ has the form $P_0 \oplus \Z G^m$ where $P_0$ is a rank one projective $\Z G$-module and $m \ge 0$. By \cite[Lemma 4.15]{Ni20a}, we have that 
\[ [(I,r) \otimes P_0] = [(I,r)]+[P_0] = [Q] \] 
and so $P_0 \oplus \Z G^m \mapsto ((I,r) \otimes P_0) \oplus \Z G^m$ defines a map $[P] \to [Q]$ which preserves the rank of the projective modules. 
To see that it is bijective note that, if $r, s \in (\Z/|G|)^\times$, then $(I,r) \otimes (I,s) \cong (I,rs)$ and so the map has inverse 
$P_0 \oplus \Z G^m \mapsto ((I,r^{-1}) \otimes P_0) \oplus \Z G^m$.
\end{proof}

In particular, if $P$ and $Q$ are both representatives of the finiteness obstruction $\sigma_k(G)$, then $[P] \cong [Q]$.
Hence, it suffices to prove Theorem \ref{thm:main-cancellation} for \textit{any} representative $P_{(G,n)}$ of the finiteness obstruction.

If $m_{\H}(G)=0$, i.e. $G$ satisfies the Eichler condition, then $[P_{(G,n)}]$ automatically has cancellation by the Swan-Jacobinski theorem \cite[Theorem 9.3]{Sw80}.

\begin{thm} \label{thm:sj}
If $G$ satisfies the Eichler condition, then $\Z G$ has projective cancellation, i.e. $[P]$ has cancellation for all projective $\Z G$-modules $P$.	
\end{thm}

If $\sigma_k(G)=0$, then $\Z G$ is a representative of the finiteness obstruction and we know already from \cite[Theorem 7.1]{Ni19} that $[\Z G]$ has cancellation if and only if $m_{\H}(G) \le 2$.
It therefore suffices to restrict our attention to those groups with $\sigma_k(G) \ne 0$ and $m_{\H}(G) \ne 0$ and so, by Theorem \ref{thm:big-table}, it remains to prove Theorem \ref{thm:main-cancellation} in the case where $G$ has type II with $m_{\H}(G) = 1$ or $2$, or type IV and $m_{\H}(G) =2$.

In order to deal with these cases, we will appeal to the following which is also a consequence of \cite[Theorem 5.1]{Ni19}.

\begin{lemma} \label{lemma:cancellation}
Let $f:G \twoheadrightarrow H$ where $G$ has periodic cohomology and $H$ is a binary polyhedral group such that $m_{\H}(G)=m_{\H}(H) \le 2$. If $P \in P(\Z G)$ is such that $f_\#([P]) \in T_H$, then $[P]$ has cancellation.
\end{lemma}

\begin{proof}
Let $r_H \in (\Z / |H|)^\times$ be such that $f_\#([P]) = [(I,r_H)]$. By Lemma \ref{lemma:swan-quotient}, the induced map $T_G \to T_H$ is surjective and so there exists $r \in (\Z /|G|)^\times$ such that $f_\#((I,r)) \cong (I,r_H)$, i.e. $r \equiv r_H$ mod $|H|$. 
By \cite{Sw60b}, there exists $P_0 \in [P]$ of rank one.
By \cite[Lemma 4.15]{Ni20a}, we have that
\[ f_\#([(I,r^{-1}) \otimes P_0]) = f_\#([(I,r^{-1})] + [P]) = [(I,r_H^{-1})]+[(I,r_H)] = 0 \in C(\Z H)\]
where we used that $[P_0]=[P]$. By \cite[Theorem I]{Sw83}, we know that $[\Z H]$ has cancellation and so $f_\#((I,r^{-1}) \otimes P_0) \cong \Z H$. Since the map $\Z H^\times \to K_1(\Z H)$ is surjective \cite[Theorems 7.15-7.18]{MOV83}, the conditions of \cite[Theorem 5.1]{Ni19} are met and so $[(I,r^{-1}) \otimes P_0]$ has cancellation. By Proposition \ref{prop:cancellation-mod-T}, this implies that $[P] = [P_0]$ has cancellation.	
\end{proof}

\subsubsection*{Type II} If $G$ has type II and $m_{\H}(G)=1$ or $2$, then $\mathcal{B}(G) =\{ Q_8\}$ or $\{Q_{16}\}$. Hence, for $k = 3$ or $4$, we have $f: G \twoheadrightarrow Q_{2^k}$ with $m_{\H}(G) = m_{\H}(Q_{2^k})$. By \cite[Theorems III,IV]{Sw83}, we have that $C(\Z Q_8) = T_{Q_8}$ and $C(\Z Q_{16}) = T_{Q_{16}}$ and so $f_\#([P_{(G,n)}]) \in T_{Q_{2^k}}$ automatically. Hence $[P_{(G,n)}]$ has cancellation, by Lemma \ref{lemma:cancellation}.

\subsubsection*{Type IV}
If $G$ has type IV, then $m_{\H}(G) =m_{\H}(\widetilde{O}) = 2$ and there exists a quotient $f: G \twoheadrightarrow \widetilde{O}$. Recall that $Q_{12} \le \widetilde{O}$ and that this is unique up to conjugacy \cite[Lemma 14.3]{Sw83}. We will need the following lemma, the proof of which is contained in the proof of \cite[Theorem 7.2]{Ni19}.
\begin{lemma} \label{lemma:O-cancellation}
If $P$ is a projective $\Z \widetilde{O}$-module, then $[P] \in T_{\widetilde{O}}$ if and only if $[\Res_{Q_{12}}^{\widetilde{O}}(P)] = 0 \in C(\Z Q_{12})$.
\end{lemma}
Let $N = \Ker(f)$ and let $H = f^{-1}(Q_{12})$ which is a subgroup of $G$ for which $N \unlhd H \le G$ and $Q_{12}=H/N$. By \cite[Lemma 7.10]{Ni19}, we have a diagram
\[
\begin{tikzcd}
C(\Z G) \arrow[r,"f_\#"] \arrow[d,"\text{\normalfont Res}^G_H"] & C(\Z \widetilde{O}) \arrow[d,"\text{\normalfont Res}^{\widetilde{O}}_{Q_{12}}"] \\
C(\Z H) \arrow[r,"(f\mid_{H})_\#"] & C(\Z Q_{12}) 
\end{tikzcd}
\]
and, by commutativity, we get that
\[ \Res_{Q_{12}}^{\widetilde{O}}(f_\#(P_{(G,n)})) = (f\mid_{H})_\#(\Res_H^G(P_{(G,n)})). \]
It follows from general properties of finiteness obstructions \cite[Remark 2.18]{DM85} that, if $[P_{(G,n)}] = \sigma_k(G) \in C(\Z G)/T_G$, then $[\Res_H^G(P_{(G,n)})] = \sigma_k(H) \in C(\Z H)/T_H$.
It is easy to see that $H$ has type I and so $\sigma_k(G)=0$ by Theorem \ref{thm:big-table}. This implies that $\Res_H^G(P_{(G,n)}) \in T_H$ and so $(f\mid_{H})_\#(\Res_H^G(P_{(G,n)})) \in T_{Q_{12}}$. However, \cite[Theorem IV]{Sw83} implies that $T_{Q_{12}}=0$ and so $[\Res_{Q_{12}}^{\widetilde{O}}(f_\#(P_{(G,n)}))]=0$ and $f_\#(P_{(G,n)}) \in T_{\widetilde{O}}$ by Lemma \ref{lemma:O-cancellation}. Hence Lemma \ref{lemma:cancellation} applies and $[P_{(G,n)}]$ has cancellation.

\section{The Eichler mass formula} \label{section:mass-formulas}

Let $K$ be a number field with ring of integers $\mathcal{O}_K$ and let $\Lambda$ be an $\mathcal{O}_K$-order in a finite-dimensional semi-simple $K$-algebra $A$.
It is a standard fact (see, for example, \cite[Lemma 2.1]{Sw80}) that, if $M$ is a finitely generated $\Lambda$-module, then $M$ is projective if and only if $M$ is locally projective, i.e. for all $p$ prime, $M_p = M \otimes \Z_p$ is projective over $\Lambda_p = \Lambda \otimes \Z_p$ where $\Z_p$ is the $p$-adic integers.

In the case where $K = \Q$, $\Lambda = \Z G$ and $A = \Q G$ for $G$ a finite group, then $M$ projective implies that $M_p$ is a free $\Z_p G$-module for all $p$ prime \cite[Theorem 2.21, 4.2]{Sw70}. In particular, in this case, $M$ is projective if and only if $M$ is locally free.

Define the \textit{locally free class group} $C(\Lambda)$ to be the equivalence classes of locally free modules up to the relation $P \simeq Q$ if $P \oplus \Lambda^i \cong Q \oplus \Lambda^j$ for some $i,j \ge 0$. By abuse of notation, we write $[P]$ to denote both the class $[P] \in C(\Lambda)$ and, where convenient, the set of isomorphism classes of projective modules $P_0$ where  $[P_0]=[P]$.

Define the \textit{class set} $\Cls\Lambda$ as the set of isomorphism classes of rank one locally free $\Lambda$-modules, which is finite by the Jordan-Zassenhaus theorem \cite[Section 24]{CR81}. 
Equivalently, this is the set of locally principal fractional $\Lambda$-ideals, under the relation $I \sim J$ if there exists $\alpha \in A^\times$ such that $I=\alpha J$ (see \cite{SV19}). This comes with the stable class map
\[ [ \,\cdot \,]_{\Lambda} : \Cls \Lambda \to C(\Lambda)\]
which sends $P \mapsto [P]$ and is surjective since every locally free $\Lambda$-module $P$ is of the form $P_0 \oplus \Lambda^i$ where $P_0 \in \Cls \Lambda$ and $i \ge 0$ \cite{Fr75}. Define $\Cls^{[P]}(\Lambda)$ to be $[\,\cdot\,]_{\Lambda}^{-1}([P])$ , i.e. the rank one locally free modules in $[P]$, and let $\SF(\Lambda)$ be $\Cls^{[\Lambda]}(\Lambda)$, i.e. the set of rank one stably free modules.

We say that $\Lambda$ has \textit{locally free cancellation} if $P \oplus \Lambda \cong Q \oplus \Lambda$ implies $P \cong Q$ for all locally free $\Lambda$-modules $P$ and $Q$. It follows from the discussion above that $\Lambda$ has locally free cancellation if and only if $[\,\cdot\,]_{\Lambda}$ is bijective, i.e. $\# \Cls \Lambda = \# C(\Lambda)$.

Similarly we say that $\Lambda$ has \textit{stably free cancellation} when $P \oplus \Lambda^i \cong \Lambda^j$ implies that $P \cong \Lambda^{j-i}$, or equivalently, if $\# \SF(\Lambda) = 1$.

If $X \subseteq \Cls \Lambda$, then we can define the \textit{mass} of $X$ to be
\[ \mass(X) = \sum_{I \in X} \frac{1}{[O_L(I)^\times : \mathcal{O}_K^\times]}.\]

Recall that a quaternion algebra $A$ over $K$ is \textit{totally definite} if $A$ is ramified over all archimedean places $\nu$, i.e. $A \otimes K_\nu$ is a division algebra over $K_\nu$. Note that every complex place $\nu$ splits since the only quaternion algebra over $\C$ is $M_2(\C)$. In particular, if $A$ is totally definite, then $K$ must be a totally real field.

Let $\zeta_K(s)$ be the Dedekind zeta function, let $h_K = |C(\mathcal{O}_K)|$ be the class number of $K$ and let $\Delta_K$ be the discriminant of $K$. The following was proven in \cite{Ei37}.

\begin{thm}[Eichler mass formula] \label{thm:EMF}
	Let $A$ be a totally definite quaternion algebra over $K$ and let $\Lambda$ be a maximal $\mathcal{O}_K$-order in $A$. If $n=[K:\Q]$, then
	\[ \mass(\Cls \Lambda) = \frac{2 \zeta_{K}(2)}{(2 \pi)^{2n}} \cdot |\Delta_K|^{3/2} \cdot h_K  \cdot \hspace{-2mm} \prod_{\mathfrak{p} \mid \disc(A)} (N_{K/\Q}(\mathfrak{p})-1).\]
\end{thm}

The following was first shown by Vign\'{e}ras in \cite{Vi76}, though a simplified proof can be found in \cite[Theorem 5.11]{SV19}.

\begin{thm} \label{thm:EMF+}
Let $A$ be a totally definite quaternion algebra over $K$ and let $\Lambda$ be a maximal $\mathcal{O}_K$-order in $A$. If $P$, $Q$ are locally free $\Lambda$-modules, then
\[ \mass(\Cls^{[P]}(\Lambda)) = \mass(\Cls^{[Q]}(\Lambda)).\]	
\end{thm}

In particular, this implies that
$ \mass(\Cls^{[P]}(\Lambda)) = \frac{\mass(\Cls \Lambda)}{|C(\Lambda)|}$,
where $C(\Lambda)$ denotes the class group of locally free $\Lambda$-modules. 

It was shown by Eichler that
\[ |C(\Lambda)| = h_K \cdot [(\mathcal{O}_K^\times)^+:(\mathcal{O}_K^\times)^2]\]
where $(\mathcal{O}_K^\times)^+$ denotes the group of totally positive units, i.e. those units $u \in \mathcal{O}_K^\times$ for which $\sigma(u) > 0$ for all embeddings $\sigma : K \hookrightarrow \R$.
The following can be shown using the results above as well as lower bounds on $|\Delta_K|^{1/[K:\Q]}$ in terms of $[K:\Q]$.

\begin{thm} \label{thm:HM06}
Let $A$ be a totally definite quaternion algebra over $K$ and let $\Lambda$ be a maximal $\mathcal{O}_K$-order in $A$. 
If $\Lambda$ has stably free cancellation, then $[K:\Q] \le 6$.
\end{thm}

\begin{remark}
This was proven by Hallouin-Maire \cite[Theorem 1]{HM06}, though it is worth noting that part of their result was incorrect as stated (see \cite{Sm15}).
\end{remark}

In the notation of \cite[Section 3]{Sw83}, define the Eichler constant
\[ \ei_{K} = \frac{ 2 \zeta_{K}(2)|\Delta_K|^{3/2}}{(2 \pi)^{2d}} =  \frac{(-1)^d \zeta_K(-1)}{2^{d-1}} \in \Q\]
where $d=[K:\Q]$ and where the second equality comes from the functional equation for $\zeta_K(s)$. This is rational since $\zeta_K(-1) \in \Q$.
Another constraint on the fields $K$ over which stably free cancellation can occur is as follows.
 
\begin{prop} \label{prop:zeta-1}
Let $A$ be a totally definite quaternion algebra over $K$ and let $\Lambda$ be a maximal $\mathcal{O}_K$-order in $A$.
If $\Lambda$ has stably free cancellation, then the numerator of $\zeta_K(-1)$ (or, equivalently, $\ei_K$) is a power of $2$.
\end{prop}

\begin{proof}
If $\Lambda$ has stably free cancellation, then $\mass(\SF(\Lambda)) = [\Lambda^\times : \mathcal{O}_K^\times]^{-1}$ since $\Lambda$ a maximal order implies $O_L(\Lambda) = \Lambda$, i.e. the numerator is $1$. By Theorems \ref{thm:EMF} and \ref{thm:EMF+}, we also have that
\[ \mass(\SF(\Lambda)) = \frac{\ei_K}{[(\mathcal{O}_K^\times)^+:(\mathcal{O}_K^\times)^2]} \cdot \prod_{\mathfrak{p} \mid \disc(A)} (N_{K/\Q}(\mathfrak{p})-1).\]
Note that $[(\mathcal{O}_K^\times)^+:(\mathcal{O}_K^\times)^2]$ is a power of $2$ since $(\mathcal{O}_K^\times)^2 \subseteq (\mathcal{O}_K^\times)^+ \subseteq \mathcal{O}_K^\times$ and $[(\mathcal{O}_K^\times)^\times:(\mathcal{O}_K^\times)^2] = 2^d$ by Dirichlet's unit theorem. 
Since $N_{K/\Q}(\mathfrak{p}) \in \Z$, this implies that the numerator of $\ei_K$, or equivalently $\zeta_K(-1)$, is a power of $2$.	
\end{proof}

\section{Orders in quaternionic components of $\Q G$} \label{section:orders-in-QG}

Recall that, for a finite group $G$, the rational group ring $\Q G$ is semisimple and so admits a decomposition into simple $\Q$-algebras. For the quaternion groups of order $4n \ge 8$, we have
\[ \Q Q_{4n} \cong \Q Q_{4n}^{\text{ab}} \times \prod_{d \ge 3, d \mid n} M_2(\Q(\zeta_d+\zeta_d^{-1})) \times \prod_{d \ge 3, d \nmid n, \, d \mid 2n} \Q[\zeta_d,j] \]
and\[ \Q Q_{4n}^{\text{ab}} \cong \begin{cases}
 \Q C_4	\cong \Q^2 \oplus \Q(i), & \text{if $n$ is odd} \\
  \Q C_2^2	\cong \Q^4, & \text{if $n$ is even}
 \end{cases}
  \]
where $Q_{4n}^{\text{ab}}$ denotes the abelianisation of $Q_{4n}$, where $\zeta_d = e^{2 \pi i/d} \in \C$ is a primitive $d$th root of unity, and $\Q[\zeta_d,j] \subseteq \H$ sits inside the real quaternions. 
This is stated on \cite[p75]{Sw83} though a more detailed proof can be found in \cite[p48-51]{Jo03a}. In order to apply the results of Section \ref{section:mass-formulas}, it will be helpful to note that
\[ \Q[\zeta_n,j] \cong \left( \frac{(\zeta_n-\zeta_n^{-1})^2,-1}{\Q(\zeta_n+\zeta_n^{-1})} \right).\]
This says that $\Q[\zeta_n,j]$ is the quaternion algebra with centre $K=\Q(\zeta_n+\zeta_n^{-1})$ and which is a free $K$-algebra of rank 4 with basis $\{1,\alpha, \beta,\alpha\beta\}$ where $\alpha, \beta \in \Q[\zeta_n,j]$ are such that $\alpha^2 = (\zeta_n-\zeta_n^{-1})^2$ and $\beta^2 = -1$.
Further details, as well as a proof, can be found in \cite[p51]{Jo03a}.
It is straightforward to check that $\Q[\zeta_n,j]$ is totally definite for $n \ge 3$ (see, for example, \cite[Lemma 4.3]{Sw83}).

If $n_i$ are distinct positive integers such that $n_i \nmid n$ and $n_i \mid 2n$ for $1 \le i \le k$, then define $\Lambda_{n_1, \cdots, n_k}$ to be the image of $\Z Q_{4n}$ under the projection of $ \Q Q_{4n}$ onto $A_{n_1, \cdots, n_k} = \prod_{i=1}^k \Q[\zeta_{n_i},j]$. For example, $\Z Q_{4n}$ projects onto $\Lambda_{2n} = \Z[\zeta_{2n},j]$ for all $n \ge 2$. 

Fix the standard presentation $Q_{4n} = \langle x, y \mid x^{n}=y^2, yxy^{-1}=x^{-1}\rangle$ and recall that, for $n \ge 3$, we have
$\Aut(Q_{4n}) =\{ \theta_{a,b} : a \in (\Z/{2n})^\times, b \in \Z/{2n}\}$
where $\theta_{a,b}(x)= x^a$ and $\theta_{a,b}(y) = x^{b}y$.

\begin{lemma}
Let $4n \ge 12$. If $n_i$ are distinct positive integers such that $n_i \nmid n$ and $n_i \mid 2n$ for $1 \le i \le k$, then the map $f: \Z Q_{4n} \twoheadrightarrow \Lambda_{n_1, \cdots, n_k}$ induces a map
\[ f_* : \Aut(Q_{4n}) \to \Aut(\Lambda_{n_1, \cdots, n_k}).\]	
\end{lemma}

\begin{proof}
It follows from \cite[p48-51]{Jo03a} that the map
$f: \Q Q_{4n} \twoheadrightarrow A_{n_1, \cdots, n_k}$
is given by $x \mapsto (\zeta_{n_1}, \cdots, \zeta_{n_k})$, $y \mapsto (j, \cdots, j)$. If $\bar{\zeta} = (\zeta_{n_1}, \cdots, \zeta_{n_k})$ and $\bar{j} = (j, \cdots, j)$, then  $\Lambda_{n_1, \cdots, n_k} = \langle \bar{\zeta},\bar{j} \rangle$ as a $\Z$-order in $A_{n_1, \cdots, n_k}$. For $a \in (\Z/{2n})^\times$ and $b \in \Z/{2n}$,  define 
\[ \bar{\theta}_{a,b} : \Lambda_{n_1, \cdots, n_k} \to \Lambda_{n_1, \cdots, n_k}\]
\[ \bar{\zeta} \mapsto \bar{\zeta}^a = (\zeta_{n_1}^a, \cdots, \zeta_{n_k}^a), \quad \bar{j} \mapsto \bar{\zeta}^b \bar{j} = (\zeta_{n_1}^bj, \cdots, \zeta_{n_k}^bj) \]
which we can extend to be a ring homomorphism. Hence $\bar{\theta}_{a,b} \in \Aut(\Lambda_{n_1, \cdots, n_k})$ and it is easy to see that $f \circ \theta_{a,b} = \bar{\theta}_{a,b} \circ f$. This implies that there is an induced map $f_* : \Aut(Q_{4n}) \to \Aut(\Lambda_{n_1, \cdots, n_k})$ where $f_*(\theta_{a,b}) = \bar{\theta}_{a,b}$.	
\end{proof}

\begin{lemma} \label{lemma:f(swan)=free}
Let $G$ be a finite group, let $r \in (\Z /|G|)^\times$ and suppose $\psi \in \Z G$ is such that $\Z G/(\psi)$ is torsion-free and $(\varepsilon(\psi),r)=1$. If $f : \Z G \twoheadrightarrow \Z G/(\psi)$, then
\[ f_\#((I,r)) \cong \Z G/(\psi).\]
\end{lemma}

\begin{proof}
First consider the exact sequence of $\Z G$-modules
\[ 0 \to (I,r) \xrightarrow[]{i} \Z G \xrightarrow[]{\varepsilon} \Z/r\Z \to 0.\]
If $\Lambda = \Z G/(\psi)$, then we can apply $\Lambda \otimes_{\Z G} -$ to the above sequence to get
\[ 0 \to \Lambda \otimes_{\Z G} (I,r) \xrightarrow[]{1 \otimes i} \Lambda \otimes_{\Z G} \Z G \xrightarrow[]{1 \otimes \varepsilon} \Lambda \otimes_{\Z G} (\Z/r\Z) \to 0\]
which is exact since $\Tor_1^{\Z G}(\Z/r\Z,\Lambda) = 0$ since $\Lambda$ is torsion-free. Since $f_\#((I,r)) \cong \Lambda \otimes_{\Z G} (I,r)$ and $\Lambda \cong \Lambda \otimes_{\Z G} \Z G$, it suffices to show that $1 \otimes i$ is an isomorphism, i.e. that $\Lambda \otimes_{\Z G} (\Z/r\Z)=0$.
Now $\Z/r\Z \cong \Z \otimes_{\Z G} (\Z/r\Z)$ and the associativity of tensor product implies that
\[ \Lambda \otimes_{\Z G} (\Z/r\Z) \cong (\Lambda \otimes_{\Z G} \Z) \otimes_{\Z G} (\Z/r\Z) \cong (\Z/\varepsilon(\psi)) \otimes_{\Z G} (\Z /r\Z) =0\]
since $r$ and $\varepsilon(\psi)$ are coprime.	
\end{proof}

The following is an extension of Theorem \ref{thm:P/Aut(G)} to this setting.

\begin{prop} \label{prop:order-quotient}
Let $f: \Z Q_{4n} \twoheadrightarrow \Lambda_{n_1, \cdots, n_k}$ where $n \ge 3$ and $n_1, \cdots, n_k$ are distinct positive integers such that $n_i \nmid n$ and $n_i \mid 2n$. If $\Aut(Q_{4n})$ acts on $[P]$ and $\bar{P} = f_\#(P)$, then there is a surjection of graded trees
\[ f_\# : [P]/\Aut(Q_{4n}) \twoheadrightarrow [\bar{P}]/\IM(f_*) \]
where the action of $\IM(f_*) \subseteq \Aut(\Lambda_{n_1, \cdots, n_k})$ on $[\bar{P}]$ is induced by $f$.
\end{prop}

\begin{proof}
Let $n_1, \cdots, n_t$ be the set of all $n_i$ such that $n_i \nmid n$ and $n_i \mid 2n$, where $k \le t$. It follows from \cite[p48-51]{Jo03a} that $\Q Q_{4n} / (y^2+1)  \cong A_{n_1, \cdots, n_t}$. In particular, $f$ factors through the map
\[ g: \Z Q_{4n} \twoheadrightarrow  \Z Q_{4n} / (y^2+1).\]
Now $\varepsilon(y^2+1)=2$ and $(2,r)=1$, and so $g_\#((I,r)) \cong \Z Q_{4n} / (y^2+1)$. Hence, by Lemma \ref{lemma:f(swan)=free}, we have that $f_\#((I,r)) \cong \Lambda_{n_1, \cdots, n_k}$. The result now follows using \cite[Corollary 8.4]{Ni20a} and a similar argument to the proof of Theorem \ref{thm:P/Aut(G)}.
\end{proof}

For the rest of this section, we will consider the cancellation problem for orders of the form $\Lambda_{n_1,\cdots,n_k}$. We begin by considering the case $k=1$.

\subsection{Cancellation for quaternionic orders}

First note that $\Lambda_{2n} = \Z[\zeta_{2n},j]$ is a $\mathcal{O}_K$-order in the quaternion algebra $\Q[\zeta_{2n},j]$ with centre $K=\Q(\zeta_{2n}+\zeta_{2n}^{-1})$. We can therefore apply the results in Section \ref{section:mass-formulas} to get:

\begin{lemma} \label{lemma:lambda-fails-cancellation}
Let $2n=16,22$ or $2n \ge 26$ with $2n \ne 30,42$.
Then $\Lambda_{2n}$ does not have stably free cancellation.	
\end{lemma}

\begin{proof}
If $\Lambda_{2n}$ has stably free cancellation, then so does $\Gamma_{2n}$ where $\Lambda_{2n} \subseteq \Gamma_{2n} \subseteq \Q[\zeta_{2n},j]$ is a maximal order by \cite[Theorem A10]{Sw83}. Since $[\Q(\zeta_{2n}+\zeta_{2n}^{-1}) : \Q] = \frac{1}{2} \varphi(2n)$, it follows from Theorem \ref{thm:HM06} that $\varphi(2n) \le 12$, i.e. if $2n \le 30$ or $2n = 36, 42$. It remains to consider the cases $2n = 16,22,26,28,36$. If $\ei_{n} = \ei_{\Q(\zeta_{n}+\zeta_{n}^{-1})}$, then Proposition \ref{prop:zeta-1} implies that the numerator of $\ei_{2n}$ is a power of $2$ for $2n = 16,22,26,28,36$. However, by \cite[Table II]{Sw83}, we have that
\[ \ei_{16} = \frac{5}{48}, \quad \ei_{22} = \frac{5}{132},\quad \ei_{26} = \frac{19}{156}, \quad\ei_{28} = \frac{13}{21},\quad \ei_{36} = \frac{31}{36} \]
which is a contradiction.	
\end{proof}

\begin{remark}
Theorem \ref{thm:HM06} and Proposition \ref{prop:zeta-1} do not characterise which orders have stably free cancellation, even among the $\Lambda_{2n}$. For example, $\Lambda_{42}$ has $[K:\Q] = 6$, $\ei_K = \frac{1}{6}$ \cite[Table II]{Sw83} but does not have stably free cancellation \cite[p88]{Sw83}.
\end{remark}

It is possible to show that $\Lambda_{2n}$ has stably free cancellation in all other cases.
It can be shown using \cite[Table 2]{SV19} that $\Lambda_{2n}$ has cancellation in all classes for all remaining cases other than $2n = 20, 24$. The cases $2n = 20,24$ can then be dealt with either using a MAGMA program, or by explicitly computing $\mass(\SF(\Lambda_{2n}))$ and showing it is equal to $[\Lambda^\times : \mathcal{O}_K^\times]^{-1}$.

We also note the following bounds which we will use in the proof of Theorem \ref{thm:main-bounds}. 

\begin{prop} \label{prop:bounds}
Let $K = \Q(\zeta_{2n}+\zeta_{2n}^{-1})$.
If $P \in \Cls(\Lambda_{2n})$, then
\[ \# \Cls^{[P]}(\Lambda_{2n}) \ge \frac{2 |\Delta_{K}|^{3/2}}{ 2^{t_{2n}} (2\pi)^{\varphi(2n)}} \ge e^{\frac{3}{8} \varphi(n) \log n +O(n \log \log n)}\]
where $t_{2n} = \ord_2([(\mathcal{O}_{K}^\times)^+:(\mathcal{O}_{K}^\times)^2])$. 
\end{prop}

\begin{proof}
To get the first inequality, note that $\# \Cls^{[P]}(\Lambda_{2n}) \ge \mass(\Cls^{[P]}(\Lambda_{2n}))$. We can then apply Theorems \ref{thm:EMF} and \ref{thm:EMF+} and note that $N_{K/\Q}(\mathfrak{p}) \ge 2$ for all $\mathfrak{p} \mid \disc(\Q[\zeta_{2n},j])$ to get that:
\[ \mass(\Cls^{[P]}(\Lambda_{2n})) \ge \frac{2 \zeta_{K}(2) |\Delta_{K}|^{3/2}}{(2\pi)^{\varphi(2n)} \cdot [(\mathcal{O}_{K}^\times)^+:(\mathcal{O}_{K}^\times)^2]} \ge \frac{2 |\Delta_{K}|^{3/2}}{ 2^{t_{2n}} (2\pi)^{\varphi(2n)}}\]
where we note that $\zeta_K(2) \ge 1$ by the Euler product formula.

Recall that $(\mathcal{O}_K^\times)^+ \subseteq \mathcal{O}_K^\times$ and so $t_{2n} \le [K:\Q] = \varphi(2n)$. In order to compute $\Delta_K = \Delta_{K/\Q}$, we will use that 
\[ \Delta_{\Q(\zeta_{2n})/\Q} = N_{K/\Q}(\Delta_{\Q(\zeta_{2n})/K}) \Delta_{K/\Q}^{[\Q(\zeta_{2n}):K]}.\] 
Since $\Delta_{\Q(\zeta_{2n})/K} = (\zeta_{2n}-\zeta_{2n}^{-1})^2$ and  $\Delta_{\Q(\zeta_{2n})/\Q} = \left( 2n \prod_{p \mid 2n} p^{\frac{-1}{p-1}}\right)^{\varphi(2n)}$, we get that
\[ \Delta_{K} = c_{2n}^{-1} \left( 2n \prod_{p \mid 2n} p^{\frac{-1}{p-1}}\right)^{\frac{1}{2}\varphi(2n)} \]
where $c_{2n} = N_{K/\Q}(\zeta_{2n}-\zeta_{2n}^{-1})$ which has $c_{2n} = 2$ if $n = 2^k$, $c_{2n} = \sqrt{p}$ if $n=p^k$ for $p$ an odd prime and $c_{2n}=1$ otherwise. If $B$ denotes the bound above, we have that:
\begin{align*} \textstyle \log B &\ge \textstyle \frac{3}{4} \varphi(2n) ( \log 2n - \sum_{ p \mid 2n} \frac{\log p}{p-1}) - \frac{3}{2} \log c_n -  \varphi(2n) \log 4 \pi -\log 2 \\
& \ge 	\textstyle \frac{3}{4} \varphi(2n) \log 2n +O(n \log \log n)
\end{align*}
since $c_{2n} \le n$ and by using that $\sum_{ p \mid 2n} \frac{\log p}{p-1} = O(\log \log n)$ by \cite[Lemma 2.7]{Sw83}. This makes sense since $\varphi(n) \ge O(n/\log \log n)$ by standard results and so 
\[ \varphi(2n)\log 2n \ge O(n \log n / \log \log n) \ge O(n \log \log n)\]
since $\log n/(\log \log n)^2 \to \infty$ as $n \to \infty$. We note finally that $\varphi(2n) \ge \frac{1}{2} \varphi(n)$ and $\log 2n \ge \log n$ to get the desired inequality.
\end{proof}

\subsection{Cancellation for higher quaternionic orders}

We will now consider the case $k \ge 2$. Firstly note that, for positive integers $n_1, \cdots, n_k$, there exists $n \ge 1$ such that $n_i \nmid n$, $n_i \mid 2n$ if and only if $\nu_2(n_i) = r$ for all $i$ and some $r \ge 1$, i.e. for all $i$, $n_i = 2^r m_i$ for some $m_i$ odd. From the definition, it is possible to show that
\[ \Lambda_{n_1,\cdots,n_k} \cong \Z Q_{4n} / (\Phi_{n_1}(x) \, \Phi_{n_2}(x) \, \cdots \, \Phi_{n_k}(x))\]
where $\Phi_n(x)$ denotes the $n$th cyclotomic polynomial. It follows that, if $S \subseteq \{n_1, \cdots, n_k\}$, then there is a quotient $\Lambda_{n_1, \cdots, n_k} \twoheadrightarrow \Lambda_S$. By \cite[Theorem A10]{Sw83} this implies that, if $\Lambda_{n_1, \cdots, n_k}$ has stably free cancellation, then $\Lambda_S$ has stably free cancellation for all $S \subseteq \{n_1, \cdots,n_k\}$. 

In particular, by Lemma \ref{lemma:lambda-fails-cancellation}, we have that $\Lambda_{n_1,\cdots,n_k}$ does not have stably free cancellation except possibly if, for all $i$, we have $n_i \le 14$ or $n_i = 18, 20, 24, 30$. By the discussion above, this is equivalent to
\[ \tag{*} \{n_1, \cdots, n_k \} \subseteq \begin{cases} \{2,6,10,14,18,30\}, & \text{if $r=1$} \\ \{4,12,20\}, & \text{if $r=2$} \\ \{8, 24\} & \text{if $r =3$}\end{cases}\]
where we note that the cases with $r \ge 4$ do not arise.

For a subset $S \subseteq \Z_{\ge 1}$, define a graph $\mathcal{G}(S)$ with vertices the elements of $S$ and where $a, b \in S$ are connected by an edge if $a/b = p^r$ for some $p$ prime and $r \in \Z$.

\begin{lemma}
Let $r \ge 1$ and let $n_1, \cdots, n_k$ be such that $\nu_2(n_i) =r$ for all $i$. If $\mathcal{G}(n_1, \cdots,n_k)$ is not connected, then there exists a splitting $\{n_1, \cdots, n_k\} = S_1 \sqcup S_2$ with $S_1$, $S_2$ non-empty such that $\Lambda_{n_1, \cdots, n_k} \cong \Lambda_{S_1} \times \Lambda_{S_2}$.
\end{lemma}

\begin{proof}
If $\mathcal{G}(n_1, \cdots, n_k)$ is not connected then, by definition, there exists a splitting $\{n_1, \cdots, n_k\} = S_1 \sqcup S_2$ with $S_1$, $S_2$ non-empty and such that there are no edges between any $a \in S_1$ and $b \in S_2$. Now consider $\Lambda_{n_1, \cdots, n_k} = \Z Q_{4n} / (\Phi_{n_1} \, \cdots \, \Phi_{n_k})$ where, for example, $n =\frac{1}{2} \text{lcm}(n_1, \cdots, n_k)$. 

If $a \in S_1$ and $b \in S_2$, then $a/b \ne p^r$ for all $r \in \Z$ and so $(\Phi_{a}, \Phi_{b})=1$ by results of Diederichsen \cite{Di40} (see also \cite[Theorem B1]{Sw83}). In particular, this implies that
\[ \textstyle (\prod_{n_i \in S_1} \Phi_{n_i}, \prod_{n_i \in S_2} \Phi_{n_i}) =1\]
and so $\Z Q_{4n} / (\Phi_{n_1} \, \cdots \, \Phi_{n_k}) \cong \Z Q_{4n} / (\prod_{n_i \in S_1} \Phi_{n_i}) \times \Z Q_{4n} / (\prod_{n_i \in S_2} \Phi_{n_i})$ which completes the proof since $\Lambda_{S_j} \cong \Z Q_{4n} / \prod_{n_i \in S_j} \Phi_{n_i}$ for $j = 1,2$.
\end{proof}

Note that, if $\Lambda_{n_1, \cdots, n_k} \cong \Lambda_{S_1} \times \Lambda_{S_2}$, then $\Lambda_{n_1, \cdots, n_k}$ has stably free cancellation if and only if $\Lambda_{S_1}$ and $\Lambda_{S_2}$ have stably free cancellation. It therefore suffices to consider the case where $\mathcal{G}(n_1, \cdots, n_k)$ is connected. 

The following cases were calculated by Swan in \cite[Lemma 8.9-8.12, 8.13]{Sw83} and \cite[Corollary 8.17, 10.14]{Sw83}.

\begin{lemma} \label{lemma:swan-examples}
Let $\Lambda = \Lambda_{2,14}$, $\Lambda_{6,18}$, $\Lambda_{6,30}$, $\Lambda_{4,12}$, $\Lambda_{4,20}$ or $\Lambda_{8,24}$. Then $\Lambda$ does not have stably free cancellation.
\end{lemma}

Let $(**)$ be the set of $\{n_1, \cdots, n_k\}$ which satisfy $(*)$ and also have no subset of the form given in Lemma \ref{lemma:swan-examples}. The following is straightforward:

\begin{lemma}
Let $r \ge 1$ and let $S=\{n_1, \cdots, n_k\}$ where $\nu_2(n_i) =r$ for all $i$. Suppose $S$ satisfies $(**)$ and $\mathcal{G}(S)$ is connected. If $k \ge 2$, then
\[ S = \begin{cases} \{2,6\}, \{2,10\}, \{2,18\}, \{10,30\}, & \text{if $k =2$} \\ \{2,6,10\}, \{2,10,18\}, \{2,10,30\}, & \text{if $k =3$} \\ \{2,10,18,30\}, & \text{if $k =4$} \end{cases} \]
and no cases with $r \ge 2$ or $k \ge 5$ arise.
\end{lemma}

It suffices to determine when stably free cancellation occurs for the orders $\Lambda_S$ where $S$ is given in the above lemma.

\begin{lemma} \label{lemma:calc}
If $\Lambda = \Lambda_{2,6}$, $\Lambda_{2,10}$ or $\Lambda_{2,18}$, then $\Lambda$ has stably free cancellation.	
\end{lemma}

\begin{proof}
The case $\Lambda_{2,2p}$ for $p$ an odd prime is dealt with in \cite[Lemma 10.13]{Sw83}, and it follows from this that $\Lambda_{2,6}$, $\Lambda_{2,10}$ have cancellation. The case $\Lambda_{2,18}$ is similar, though with added complications. First recall that $\Lambda_{2,18} = \Z Q_{36}/(\Phi_2 \, \Phi_{18})$. If $I = (\Phi_2)$, $J = (\Phi_{18})$ are ideals in $\Z Q_{36}$, then $I \cap J = (\Phi_2 \, \Phi_{18})$ and $I + J = (\Phi_2, \Phi_{18}) = (\Phi_2,3)$ by \cite[Theorem B1]{Sw83}. By \cite[Example 42.3]{CR87}, we get the following Milnor square
\[
\begin{tikzcd}
	\Lambda_{2,18} \ar[d] \ar[r] & \Z[\zeta_{18},j] \ar[d] \\
	\Z[j] \ar[r] & \F_3[j]
\end{tikzcd}
\]
since $\Lambda_{18} = \Z[\zeta_{18},j]$ and $\Lambda_2 = \Z[j]$. If $\Ker$ denotes the fibre over $(\Lambda_2, \Lambda_{18})$, then
\begin{align*}	
\Ker(\Cls(\Lambda_{2,18}) \to \Cls(\Lambda_{2}) \times \Cls(\Lambda_{18})) &\cong \Z[j]^\times \backslash \, \F_3[j]^\times \slash \Z[\zeta_{18},j]^\times = \{ [1], [1+j]\} \\
\Ker(C(\Lambda_{2,18}) \to C(\Lambda_{2}) \times C(\Lambda_{18})) &\cong \frac{K_1(\F_3[j])}{K_1(\Z[j]) \times K_1(\Z[\zeta_{18},j])} \cong \F_3^\times / (\F_3^\times)^2
\end{align*}
where the first identifications are standard \cite{Sw80} and the second comes from \cite[Lemma 7.5,7.6]{MOV83} since $h_9$ is odd. Since both sets have the same size, and $\Z[j]$, $\Z[\zeta_{18},j]$ have cancellation, it follows that $\Lambda_{2,18}$ has stably free cancellation.	
\end{proof}

\begin{lemma}
Let $n, m \ge 2$ even such that $n/m \ne p^r$ for all $p$ prime and $r \in \Z$. If $\Lambda_{2,n}$, $\Lambda_{2,m}$ have stably free cancellation, then so does $\Lambda_{2,n,m}$.	
\end{lemma}

\begin{proof}
Let $I = ( \Phi_2 \, \Phi_n)$ and $J = (\Phi_2, \, \Phi_m)$ so that $I \cap J = (\Phi_2 \, \Phi_n \, \Phi_m)$ and $I+J = (\Phi_2 \, \Phi_n,  \Phi_2 \, \Phi_m) = (\Phi_2)$ since $(\Phi_n,  \Phi_m) =1$ by \cite[Theorem B1]{Sw83}. By \cite[Example 42.3]{CR87}, we get the following Milnor square
\[
\begin{tikzcd}
	\Lambda_{2,n,m} \ar[d] \ar[r] & \Lambda_{2,m} \ar[d] \\
	\Lambda_{2,n} \ar[r] & \Z[j]
\end{tikzcd}
\]
since $\Lambda_2 = \Z[j]$. Since $\Z[j]^\times=\{ \pm 1, \pm j\}$, we have that $\Lambda_{2,n}^\times \twoheadrightarrow \Z[j]\times$ since the units $\pm 1, \pm j \in \Lambda_{2,n}$ lift from $\Lambda_2$, $\Lambda_n$. If $\Ker$ is the fibre over $(\Lambda_{2,n}, \Lambda_{2,m})$, then
\[
\Ker(\Cls(\Lambda_{2,n,m}) \to \Cls(\Lambda_{2,n}) \times \Cls(\Lambda_{2,m})) \cong \Lambda_{2,n}^\times \, \backslash \Z[j]^\times \slash \Lambda_{2,m}^\times = 1 \]
from which the result follows easily.
\end{proof}

In particular, since $\Lambda_{2,n}$ has stably free cancellation for $n=6,10,18$ by Lemma \ref{lemma:calc}, this implies that $\Lambda_{2,6,10}$, $\Lambda_{2,10,18}$ have stably free cancellation.

\begin{lemma}
$\Lambda_{10,30}$ does not have stably free cancellation.	
\end{lemma}

\begin{proof}
As above, consider the Milnor square corresponding to the ideals $I = (\Phi_{10})$ and $J = (\Phi_{30})$ which has $I+J = (\Phi_{10},3)$ by \cite[Theorem B1]{Sw83}:
\[
\begin{tikzcd}
	\Lambda_{10,30} \ar[d] \ar[r] & \Z[\zeta_{30},j] \ar[d] \\
	\Z[\zeta_{10},j] \ar[r] & \F_3[\zeta_{10},j]
\end{tikzcd}
\]
By \cite[p76]{Sw83}, $\Lambda_{10}$ is unramified at 3 and so $\F_3[\zeta_{10},j] \cong M_2(\F_3[\zeta_{10}+\zeta_{10}^{-1}])$. 
Since $\zeta_{10}+\zeta_{10}^{-1} = \frac{1}{2}(1+\sqrt{5})$ and $\F_3(\sqrt{5}) \cong \F_9$, we get $K_1(\F_3[\zeta_{10},j]) \cong K_1(M_2(\F_9)) \cong \F_9^\times$ by Morita equivalence. Consider the composition
\[ f: \Z[\zeta_{30},j]^\times \to K_1(\Z[\zeta_{30},j]) \to K_1(\F_3[\zeta_{10},j]) \cong \F_9^\times.\]
Since 30 is composite, $1+ \zeta_{30} \in \Z[\zeta_{30}]^\times$. By the argument given in the proof on \cite[Lemma 8.9]{Sw83}, we get that
\[ 1+\zeta_{30} \mapsto 2 + (\zeta_{30}+\zeta_{30}^{-1}) \mapsto 2+(\zeta_{10}+\zeta_{10}^{-1}) = 2+\frac{1}{2}(1+\sqrt{5}) = 1 - \sqrt{5} \in \F_9^\times.\]
Since $1-\sqrt{5}$ is a generator of $\F_9 \cong \F_3[\sqrt{5}]$, this implies that $f$ is surjective. In particular, this implies that
\[ \Ker(C(\Lambda_{10,30}) \to C(\Lambda_{10}) \times C(\Lambda_{30})) \cong \frac{K_1(\F_3[\zeta_{10},j])}{K_1(\Z[\zeta_{10},j]) \times K_1(\Z[\zeta_{30},j])} =1.\]
It therefore suffices to show that $\Z[\zeta_{10},j]^\times \, \backslash \F_3[\zeta_{10},j]^\times \slash \Z[\zeta_{30},j]^\times \ne 1$. 
By \cite[Lemma 7.5,7.6]{MOV83}, we have that $\Z[\zeta_{2n},j]^\times = \langle \Z[\zeta_{2n}]^\times,j \rangle$ for all $n$ and so
\[ \IM(\Z[\zeta_{10},j]^\times \times \Z[\zeta_{30},j]^\times \to \F_3[\zeta_{10},j]^\times) \cap \F_3[j] = \{ \pm 1, \pm j\}.\]
It follows that $[1] \ne [1+j]$, which completes the proof.
\end{proof}

Combining all these results together leads to the following, which determines precisely when stably free cancellation occurs for the $\Lambda_{n_1. \cdots, n_k}$.

\begin{thm} \label{thm:higher-orders}
Let $n_1, \cdots, n_k$ be positive integers such that, for some $n$, we have $n_i \nmid n$ and $n_i \mid 2n$. Then a complete list of the $\Lambda_{n_1 , \cdots, n_k}$ which have stably free cancellation is as follows:
\begin{enumerate}[\normalfont (i)]
\item $\Lambda_{2}$, $\Lambda_{4}$, $\Lambda_{6}$, $\Lambda_{8}$, $\Lambda_{10}$, $\Lambda_{12}$, $\Lambda_{14}$, $\Lambda_{18}$, $\Lambda_{20}$, $\Lambda_{24}$, $\Lambda_{30}$
\item $\Lambda_{2,6}$, $\Lambda_{2,10}$, $\Lambda_{2,18}$, $\Lambda_{2,30}$, $\Lambda_{6,10}$, $\Lambda_{6,14}$, $\Lambda_{10,14}$, $\Lambda_{10,18}$, $\Lambda_{14,18}$, $\Lambda_{14,30}$, $\Lambda_{18,30}$
\item $\Lambda_{2,6,10}$, $\Lambda_{2,10,18}$, $\Lambda_{2,18,30}$, $\Lambda_{6,10,14}$, $\Lambda_{10,14,18}$, $\Lambda_{14,18,30}$.
\end{enumerate}
\end{thm}

\begin{proof}
Consider the not-necessarily-connected subsets of the form $(*)$. If $k=1$, then use Lemma \ref{lemma:lambda-fails-cancellation}. By the work above, if $k \ge 2$, then eliminate all subsets which contain $\{2,14\}$, $\{4,12\}$, $\{4,20\}$, $\{6,18\}$, $\{6,30\}$, $\{8,24\}$, $\{10,30\}$.
\end{proof}

It is not much more difficult to construct a list of those orders $\Lambda_{n_1 , \cdots, n_k}$which have cancellation in all classes, as well as those orders which have non-cancellation in all locally free classes. However, we will omit this here for brevity.

\section{Proof of Theorem \ref{thm:main}} \label{section:non-cancellation}

Let $G$ have $k$-periodic cohomology, let $n=ik$ or $ik-2$ for some $i \ge 1$ and let $P_{(G,n)} \in P(\Z G)$ be a representative for the Swan finiteness obstruction $\sigma_{ik}(G) \in C(\Z G)/T_G$. We will now proceed to prove the following theorem which, by \cite[Theorem B]{Ni20a}, is equivalent to Theorem \ref{thm:main}.

\begin{thm} \label{thm:main-non-cancellation}
$[P_{(G,n)}]/\Aut(G)$ has cancellation if and only if $m_{\H}(G) \le 2$.
\end{thm}
	
If $m_{\H}(G) \le 2$, then Theorem \ref{thm:main-cancellation} implies that $[P_{(G,n)}]$ has cancellation and so $[P_{(G,n)}]/\Aut(G)$ has cancellation also. It therefore remains to show that $m_{\H}(G) \ge 3$ implies $[P_{(G,n)}]/\Aut(G)$ has non-cancellation.

Let $G$ have periodic cohomology and $m_{\H}(G) \ge 3$. By Lemma \ref{lemma:MH>=3}, there is a quotient $f: G \twoheadrightarrow Q_{4m}$ for some $m \ge 6$. Suppose $g: \Z Q_{4m} \twoheadrightarrow \Lambda_{n_1, \cdots, n_k}$ and $n_1, \cdots, n_k$ are distinct positive integers such that $n_i \nmid m$ and $n_i \mid 2m$. Then, by combining Corollary \ref{cor:P/Aut(G)-periodic} and Proposition \ref{prop:order-quotient}, there is a surjection of graded trees
\[ f_\# : [P_{(G,n)}]/\Aut(G) \twoheadrightarrow [\widebar{P_{(G,n)}}]/\IM((g \circ f)_*)\]
where $\widebar{P_{(G,k)}} = (g \circ f)_\#(P_{(G,n)})$, $(g \circ f)_*: \Aut(G) \to \Aut(\Lambda_{n_1, \cdots, n_k})$ and the action of $\IM((g \circ f)_*) \subseteq \Aut(\Lambda_{n_1, \cdots, n_k})$ on $[\bar{P}]$ is induced by $g \circ f$.

The following is immediate from Theorem \ref{thm:higher-orders}. Note that this shows that $\Z Q_{4m}$ does not have stably free cancellation for $m \ge 6$ as was shown in \cite[Theorem I]{Sw83}.

\begin{lemma} \label{lemma:q4n-quotient}
If $m \ge 6$, then $Q_{4m}$ has a quotient of the form $\Lambda = \Lambda_{n_1,\cdots,n_k}$ with $n_i \nmid m$, $n_i \mid 2m$ which does not have stably free cancellation. In particular,
\begin{enumerate}[\normalfont(i)]
\item If $m \ne 6,7,9,10,12,15$, then we can take $\Lambda = \Lambda_{2m}$
\item If $m = 6,7,9,10,12,15$, then we can take $\Lambda = \Lambda_{4,12}$, $\Lambda_{2,14}$, $\Lambda_{6,18}$, $\Lambda_{4,20}$, $\Lambda_{8,24}$, $\Lambda_{6,30}$ and $\Lambda_{6,42}$ respectively.
\end{enumerate}
\end{lemma}

Recall that, if $\Lambda \subseteq A$ is a $\Z$-order in a finite-dimensional semisimple separable $\Q$-algebra $A$ and $\Lambda \subseteq \Gamma \subseteq A$ is a maximal order, then the \textit{kernel group} is defined as $D(\Lambda) = \Ker(i_* : C(\Lambda) \to C(\Gamma))$ where $i: \Lambda \hookrightarrow \Gamma$, and note that $i_*$ is surjective by \cite[Theorem A10]{Sw83}. By \cite[Theorem A24]{Sw83}, this is independent of the choice of $\Gamma$ and, if $f: \Lambda_1 \to \Lambda_2$ is a map of $\Z$-orders, then $f$ induces a map $f_* : D(\Lambda_1) \to D(\Lambda_2)$.

It was shown by Milgram \cite[Theorem 2.B.1]{Mi85} that $[P_{(G,n)}] \in D(\Z G)$ and so $[\widebar{P_{(G,n)}}] \in D(\Lambda_{n_1, \cdots, n_k})$ by the remarks above. In particular, if $D(\Lambda_{n_1, \cdots, n_k})=0$, then we have that $[\widebar{P_{(G,n)}}] = 0 \in C(\Lambda_{n_1, \cdots, n_k})$ automatically.

\begin{lemma} \label{lemma:D=0}
Suppose $\Lambda = \Lambda_{n_1, \cdots, n_k} \subseteq A_{n_1, \cdots, n_k}$ is of one of the following forms:
\begin{enumerate}[\normalfont (i)]
\item $\Lambda = \Lambda_{2n}$ for $n \ge 1$
\item $\Lambda = \Lambda_{4,12}$, $\Lambda_{4,20}$, $\Lambda_{8,24}$, $\Lambda_{6,30}$, $\Lambda_{6,42}$.
\end{enumerate}
Then $D(\Lambda) = 0$.
\end{lemma}

\begin{proof}
By \cite[Corollary 8.3]{Sw83}, we have that $D(\Lambda_{2n})=0$. The other orders are of the form $\Lambda_{d,2n}$ for some $d \ne 2n$ such that $d \nmid n$ and $d \mid 2n$. Suppose that $d$ and $n$ are one of the give pairs above. It follows from the proofs of \cite[Lemma 8.9-8.12, 8.14]{Sw83} that the projection $f: \Lambda_{d,2n} \to \Lambda_{2n}$ induces an isomorphism $f_*: C(\Lambda_{d,2n}) \cong C(\Lambda_{2n})$. Since $D(\Lambda_{2n})=0$, this implies that $C(\Lambda_{d,2n}) \cong C(\Gamma)$ for any maximal order $\Lambda_{2n} \subseteq \Gamma \subseteq A_{2n}$ by \cite[Theorem A24]{Sw83}.

Let $\Lambda_{d,2n} \subseteq \Gamma_{d,2n} \subseteq A_{d,2n}$ be a maximal order. Since $A_{d,2n} \cong A_d \times A_n$, we must have that $\Gamma_{2n,d} \cong \Gamma_d \times \Gamma_{2n}$ where $\Gamma_d \subseteq A_d$ and $\Gamma_{2n} \subseteq A_{2n}$ are maximal orders \cite[p33]{Jo03a}. Since $\Gamma \twoheadrightarrow \Gamma_{2n}$ induces a surjection $C(\Gamma) \twoheadrightarrow C(\Gamma_{2n})$ \cite[Theorem A10]{Sw83}, we have that
$|C(\Gamma_{2n})| = |C(\Lambda_{d,2n})| \ge |C(\Gamma_{d,2n})| \ge |C(\Gamma_{2n})|$
by considering the case $\Gamma = \Gamma_{2n}$. This implies that $D(\Lambda_{d,2n})=0$.
\end{proof}

\begin{remark}
Lemma \ref{lemma:D=0} cannot be extended to $\Lambda_{2,14}$ or $\Lambda_{6,18}$ or any other orders arising from $Q_{28}$, $Q_{36}$ which do not have stably free cancellation. In particular
\[ \Q Q_{28} \cong \Q D_{14} \times A_{2,14}, \quad \Q Q_{36} \cong \Q D_{18} \times A_{2,6,18}\]
and so the only options are $\Lambda_2$, $\Lambda_6$, $\Lambda_{14}$, $\Lambda_{18}$, $\Lambda_{2,6}$, $\Lambda_{2,18}$ which all have stably free cancellation, or $\Lambda_{2,14}$, $\Lambda_{6,18}$, $\Lambda_{2,6,18}$ which all do not have stably free cancellation by Theorem \ref{thm:higher-orders} but have $D(\Lambda) \ne 0$ from results in \cite{Sw83}.
\end{remark}

Suppose $G$ has a quotient $Q_{4m}$ for some $m \ge 6$ and $4m \ne 28, 36$. Then Lemmas \ref{lemma:q4n-quotient} and \ref{lemma:D=0} imply that there exists $\Lambda_G = \Lambda_{n_1, \cdots,n_k}$ such that $D(\Lambda_G) = 0$ and $[\Lambda_G]$ has non-cancellation. Since $[P_{(G,n)}] \in D(\Z G)$, this implies that $[\widebar{P_{(G,n)}}] = 0 \in C(\Lambda_G)$. 
If $4m = 28, 36$, then Lemma \ref{lemma:q4n-quotient} implies that there exists $\Lambda_G = \Lambda_{n_1, \cdots,n_k}$ such that $[\Lambda_G]$ has non-cancellation. Since $Q_{28}$ and $Q_{36}$ have type I, Proposition \ref{prop:type-closed} implies that $G$ has type I and so $\sigma_k(G)=0$ by Theorem \ref{thm:big-table}. Hence $[P_{(G,n)}] \in T_G$ and so similarly we have that $[\widebar{P_{(G,n)}}] = 0 \in C(\Lambda_G)$ by Lemma \ref{lemma:f(swan)=free}.

Hence, in both cases, there is a surjection of graded trees
\[ [P_{(G,n)}]/\Aut(G) \twoheadrightarrow [\Lambda_G]/\Aut(\Lambda_G).\]
By, for example \cite[Lemma 8.1]{Ni20a}, the action of $\Aut(\Lambda_G)$ on $[\Lambda_G]$ fixes the free module. Since $[\Lambda_G]$ has non-cancellation, this now implies that $[\Lambda_G]/\Aut(\Lambda_G)$ has non-cancellation. To see this, note that $[\Lambda_G]$ having non-cancellation is equivalent to the existence of a non-free stably free $\Lambda_G$-module $S$. If $S$ and $\Lambda_G$ coincide in $[\Lambda_G]$ under the orbit of the $\Aut(\Lambda_G)$ action, then $S \cong f \cdot \Lambda_G$ for some $f \in \Aut(\Lambda_G)$. However, since the action of $\Aut(\Lambda_G)$ on $[\Lambda_G]$ fixes the free module, we have that $f \cdot \Lambda_G \cong \Lambda_G$ which implies that $S \cong \Lambda_G$ which is a contradiction. Hence $[\Lambda_G]/\Aut(\Lambda_G)$ had non-cancellation.
This completes the proof of Theorem \ref{thm:main-non-cancellation} and hence completes the proof of Theorem \ref{thm:main}.

We conclude this section by picking back up on Remark \ref{remark:thm-main} (b) in which we referred to the following mild generalisation of Theorem \ref{thm:main} in the case $n=2$.
Let $\text{D2}(G)$ (resp. $\text{PD2}(G)$) denote the set of homotopy (resp. polarised homotopy) types of finite D2-complexes $X$ with $\pi_1(X) \cong G$. Then the following holds in the case $n=2$ without the assumption that $G$ has the D2 property.

\begin{thm} \label{thm:main-D2}
Let $G$ have $4$-periodic cohomology. Then the following are equivalent:
\begin{enumerate}[\normalfont(i)]
\item $\text{\normalfont D2}(G)$ has cancellation
\item $\text{\normalfont PD2}(G)$ has cancellation
\item $m_{\H}(G) \le 2$.
\end{enumerate}
\end{thm}

\begin{proof}
The proof of Theorem \ref{thm:main} extends in a straightforward manner to Theorem \ref{thm:main-D2} in light of \cite[Remark 5.2 (a)]{Ni20a} (see also \cite[Theorem 2.1]{Ni19}). This states that there is an isomorphism of graded trees 
\[ \Psi : \text{PD2}(G) \to [P_{(G,2)}]\] 
where $P_{(G,2)}$ is a representative of the Swan finiteness obstruction. It follows easily that there is also an isomorphism of graded trees 
\[ \bar{\Psi} : \text{D2}(G) \to [P_{(G,2)}]/\Aut(G).\]
The theorem is now a consequence of the algebraic results established in Theorems \ref{thm:main-cancellation} and \ref{thm:main-non-cancellation}.
\end{proof}

\section{Proof of Theorem \ref{thm:main-bounds}}
\label{section:bounds}

Recall that, if $G$ is finite and $n$ is even, then $\HT(G,n)$ is a fork in that it has a single vertex at each non-minimal height and finitely many at the minimal level \cite[Corollary 4.7]{Ni20a}. Let \[ N(G,n) = \#\{X \in \HT(G,n) : \vv \chi(X) \text{ is minimal}\}.\] 

We will now prove the following, which is a restatement of Theorem \ref{thm:main-bounds}.

\begin{thm}
Let $G$ have $k$-periodic cohomology and let $n$ be such that $k \mid n$ or $n+2$ and, if $n=2$, suppose $G$ has the {\normalfont D2} property. If $m = m_{\H}(G)$, then
\[ N(G,n) \ge e^{\tfrac{m \log m}{8 \log \log m} + O(m \log \log m)}. \]
\end{thm}

\begin{proof}
By \cite[Theorem B]{Ni20a}, there is a bijection
\[ \HT(G,n) \cong [P_{(G,n)}]/\Aut(G)\]
for some $P_{(G,n)} \in P(\Z G)$ such that $[P_{(G,n)}] = \sigma_n(G) \in C(\Z G)/T_G$.
Now suppose $m = m_{\H}(G) \ge 3$. By Lemma \ref{lemma:MH>=3}, there is a quotient $f: G \twoheadrightarrow Q_{4m_0}$ for some $m_0 \ge 2m/3$. By Corollary \ref{cor:P/Aut(G)-periodic}, this implies that
\[ N(G,n) \ge \# \Cls^{[P_{(G,n)}]}(\Z G)/\Aut(G) \ge \# \Cls^{[P]}(\Z Q_{4m_0}) /\IM(f_*)\]
where $P = f_\#(P_{(G,n)})$. Since there is a quotient of $\Z$-orders $g: Q_{4m_0} \twoheadrightarrow \Lambda_{2m_0}$, it follows from \cite[Theorem A10]{Sw83} that 
\[ \# \Cls^{[P]}(\Z Q_{4m_0}) \ge \# \Cls^{[\bar{P}]}(\Lambda)\] 
where $\bar{P} = g_\#(P)$. Since $\bar{P} = (g \circ f)_\#(P_{(G,n)})$, we have that $[\bar{P}] = (g \circ f)_\#([P_{(G,n)}]) \in C(\Lambda_{2m_0})$. It follows from \cite[Theorem 2.B.1]{Mi85} that $[P_{(G,n)}] \in D(\Z G)$ and so $[\bar{P}] \in D(\Lambda_{2m_0})$. By \cite[Corollary 8.3]{Sw83}, we have that $D(\Lambda_{2m_0})=0$ and so $[\bar{P}] = [\Lambda_{2m_0}]$.

Since $|\Aut(Q_{4m_0})| = 2m_0 \varphi(2m_0)$, Proposition \ref{prop:bounds} implies that
\[ N(G,n) \ge \frac{1}{2m_0 \varphi(2m_0)} e^{\frac{3}{8} \varphi(m_0) \log m_0 + O(m_0 \log \log m_0)}\]
and we can omit $1/2m_0 \varphi(2m_0)$ since it is sub-exponential.
By \cite[Theorem 328]{HW60}, we have that $\varphi(n) \ge n / 2\log \log n$ for $n$ sufficiently large. Since $m_0 \ge 2m/3$, this implies that
\[ \log N(G,n) \ge \frac{3}{16} \cdot \frac{(2m/3) \log (2m/3)}{\log \log (2m/3)} + O(m \log \log m) \ge \frac{m \log m}{8 \log \log m} + O(m\log \log m)\]
since $\log x/ \log \log x$ is increasing for $x$ sufficiently large.	
\end{proof}

We note also the following improvement in the case where $G$ is a quaternion group, the proof of which is contained in the argument above.

\begin{prop}
Let $n$ be even and, if $n=2$, assume $Q_{4m}$ has the {\normalfont D2} property. Then we have
\[ N(Q_{4m},n) \ge e^{\frac{3}{8} \varphi(m) \log m + O(m \log \log m)}.\]
\end{prop}

In particular, if $p$ is prime and $m =m_{\H}(Q_{4p}) = \lfloor p/2 \rfloor$, then
\[ N(Q_{4p},n) \ge e^{\frac{3}{8} p \log p + O(p \log \log p)} \ge e^{\frac{3}{4} m \log m + O(m \log \log m)}. \]

\section{Applications to group presentations}
\label{section:D2}

We will adopt the convention that a finite presentation $\mathcal{P} = \langle x_1,\cdots,x_n \mid r_1 = r_1' , \cdots, r_m = r_m' \rangle$, defined in terms of relations, is identified with the presentation $\mathcal{P} = \langle x_1,\cdots,x_n \mid r_1(r_1')^{-1} , \cdots, r_m(r_m')^{-1} \rangle$ defined in terms of relators. This allows us to associate to $\mathcal{P}$ a canonical presentation complex $X_{\mathcal{P}}$. We say that two presentations $\mathcal{P}$ and $\mathcal{Q}$ are \textit{homotopy equivalent}, written $\mathcal{P} \simeq \mathcal{Q}$, if $X_{\mathcal{P}} \simeq X_{\mathcal{Q}}$ are homotopy equivalent as spaces. 

Given a finite presentation $\mathcal{P} = \langle x_1,\cdots,x_n \mid r_1 , \cdots, r_m \rangle$ for a group $G$, we can construct another presentation for $G$ by applying one of the following operations to $\mathcal{P}$:
\begin{clist}{(A)}
\item Replace a relator $r_i$ by $\omega r_i \omega^{-1}$ for some word $\omega\in F(x_1,\cdots,x_n)$. (\textit{Conjugation})
\item For $i \ne j$, replace a relator $r_i$ by $r_ir_j$. (\textit{Right-multiplication})
\item Replace a relator $r_i$ by $r_i^{-1}$. (\textit{Inversion})
\end{clist}

The operations (A)-(C) are known as \textit{$Q$-transformations} and we say that two presentations $\mathcal{P}$ and $\mathcal{Q}$ are \textit{$Q$-equivalent}, written $\mathcal{P} \simeq_{Q} \mathcal{Q}$ if they are related by a sequence of $Q$-transformations. 
The following are consequences of the operations (A)-(C) and so induce $Q$-equivalences \cite[Footnote 1]{Me76}:
\begin{clist}{(A)}
\setcounter{enumi}{3}
\item For $i \ne j$, replace a relator $r_i = a \cdot b$ with $a \cdot r_j \cdot b$ for words $a,b \in F(x_1,\cdots,x_n)$ (\textit{Insertion})
\item 
For $i \ne j$, replace $r_i$ with $r_j$ and $r_j$ with $r_i$ (\textit{Permutation})
\end{clist}

If $\mathcal{P} \simeq_{Q} \mathcal{Q}$, then $\mathcal{P} \simeq \mathcal{Q}$ \cite[p8]{Me76}. Finding $Q$-equivalences between presentations is therefore a convenient way to prove that two presentations are homotopy equivalent. Note that the converse is far from true and there are many intermediate equivalence relations (see \cite[Chapters III \& XII]{HMS93}).

\begin{remark}
We previously defined the homotopy type of a presentation defined by relations $r_i = r_i'$ by first viewing it as a presentation defined by relators $r_i(r_i')^{-1}$. One nice consequence of the discussion above is that the resulting homotopy type is independent of the part of the word we `cut' the equality by to turn it into a relator since different choices are related by operations of type (A) and (C).
\end{remark}

Fix $n \ge 2$ and let $\mathcal{P}_{\std} = \langle x , y \mid x^n = y^2, y^{-1}xy = x^{-1} \rangle$ denote the standard presentation for $Q_{4n}$. For $r \in \Z$, consider the presentations
\[ \mathcal{P}_{\MP}^r = \langle x , y \mid x^n = y^2, y^{-1}xyx^{r-1} = x^ry^{-1}x^2y \rangle  \]
which were mentioned in Section \ref{ss:D2} of the introduction. 
We will now prove the following, which was claimed in the introduction. 
Note that, in statement (ii), there is no assumption that $\mathcal{P}_{\MP}^r$ is a presentation for $Q_{4n}$.

\begin{prop} \label{prop:MP}
\begin{clist}{(i)}
\item
If $r=0$ or $1$, then $\mathcal{P}_{\MP}^r \simeq \mathcal{P}_{\std}$.
\item
If $r \equiv s \mod n$, then $\mathcal{P}_{\MP}^r \simeq \mathcal{P}_{\MP}^s$.
\end{clist}
\end{prop}

\begin{proof}
In each case, we will show that the pairs of presentations are $Q$-equivalent.

(i) If $r=0$, the second relator becomes $(y^{-1}xyx^{-1})(y^{-1}x^2y)^{-1}$. After conjugation, this becomes $yx^{-1}y^{-1}x^{-1}$ which simplifies to $yxy^{-1}x$ after inversion and further conjugation. The case $r=1$ is similar.

(ii) For $r \in \Z$, we will show that $\mathcal{P}_{\MP}^r$ and $\mathcal{P}_{\MP}^{r+n}$ are $Q$-equivalent. The second relator of $\mathcal{P}_{\MP}^r$ is of the form $a x^r b x^{-r}$ for some $a,b\in F(x,y)$. Using insertion, we can replace this with $ax^{r+n}y^{-2}bx^{-r}$. By insertions and inversions, we can replace $y^{-2}$ with $x^n$ and vice versa. By doing this, we can move $y^{-2}$ past $b$ to put the second relator in the form $ax^{r+n}by^{-2}x^{-r}$. Finally, we apply insertion and inversion once more to replace this with $ax^{r+n}bx^{-(r+n)}$ which is the second relator of $\mathcal{P}_{\MP}^{r+n}$.
\end{proof}

The remainder of this section will be devoted to the proof of Corollary \ref{cor:main-D2}. We will begin with the following, which is a simple counting argument. Recall that an \textit{$(n,m)$-presentation} is a presentation with $n$ generators and $m$ relators.

\begin{lemma} \label{lemma:upper-bound-on-number-of-pres}
Let $t \ge 1$ be an integer. Then:
\[\#\{\mathcal{P} : \mathcal{P} \text{ a $(2,2)$-presentation, $\ell(\mathcal{P}) \le t$}\} \le e^{6t}.\]
\end{lemma}

\begin{proof}
Each presentation $\mathcal{P}$ has the form
$\mathcal{P} = \langle x, y \mid r,s \rangle$ for $r,s \in F_2 =F(x,y)$.
Hence we have
\begin{align*} &\#\{\mathcal{P} : \mathcal{P} \text{ a $(2,2)$-presentation, $\ell(\mathcal{P}) \le t$}\} = \#\{(r,s) \in F_2^2 : |r|+|s| \le t\} \\ 
	&= \sum_{i=0}^t \sum_{j=0}^{i} \#\{(r,s) \in F_2^2 : |r|=j, |s|=i-j\} \le \sum_{i=0}^t \sum_{j=0}^{i} 4^j 4^{i-j} \le (t+1)^2 4^t
\end{align*}
where we used that $\#\{r \in F_2 : |r| =k\} \le 4^k$ since each letter is of the form $x,x^{-1},y,y^{-1}$.

Finally, note that $\log ((t+1)^24^t) = t \log 4 + 2\log(t+1) \le t(\log 4 + 4) \le 6t$.
\end{proof}

We will now prove the following which implies Corollary \ref{cor:main-D2} by letting $f(n)=n$ be the constant function and letting $G = Q_{4n}$ so that $m_{\H}(G) = \lfloor n/2 \rfloor$.

\begin{thm} \label{thm:main-D2-detailed}
Let $\lambda > 0$ and let $f : \N \to \R_{> 0}$ be a function for which $n \log n / (f(n) \log\log n) \to \infty$ as $n \to \infty$. Then there exists a constant $N_{\lambda}$ such that, if $G$ has $4$-periodic cohomology and $m = m_{\H}(G) > N_{\lambda}$, then at least one of the following holds:
\begin{clist}{(i)}
\item
$G$ does not have the {\normalfont D2} property.
\item
$G$ has a balanced presentation which is not homotopy equivalent to a $(2,2)$-presentation $\mathcal{P}$ with $\ell(\mathcal{P}) \le \lambda f(m)$.
\end{clist}
\end{thm}

\begin{proof}
Fix $\lambda > 0$ and a function $f : \N \to \R_{>0}$ with the desired property.
Let $G$ have 4-periodic cohomology and suppose $G$ does not satisfy conditions (i) or (ii). It suffices to show that $m = m_{\H}(G)$ is bounded

Since $G$ does not satisfy (i), it has the D2 property. By \cite[p6532]{Ni19}, this implies that $\Def(G) = 0$ and so $N(G,2)$ is the number of balanced presentations up to homotopy equivalence. On the other hand, since $G$ does not satisfy (ii), every balanced presentation is homotopy equivalent to a $(2,2)$-presentation $\mathcal{P}$ with $\ell(\mathcal{P}) \le \lambda f(m)$. This implies that
\[ N(G,2) \le \#\{\mathcal{P} : \mathcal{P} \text{ a $(2,2)$-presentation, $\ell(\mathcal{P}) \le \lambda f(m)$}\} \le e^{6\lambda f(m)} \]
where the second inequality is by Lemma \ref{lemma:upper-bound-on-number-of-pres}.

Since $G$ has the D2 propery, Theorem \ref{thm:main-bounds} implies that there is a constant $C > 0$ for which 
\[ N(G,2) \ge e^{C \tfrac{m \log m}{\log\log m}} \]
and so $m \log m /(f(m) \log \log m) \le 6 \lambda/C$. By assumption, we have that $n \log n / (f(n) \log\log n) \to \infty$ as $n \to \infty$. Since $6 \lambda/C$ is constant, this implies that $m$ is bounded.
\end{proof}

\section{Examples} \label{section:examples}

Let $G$ have $k$-periodic cohomology, let $n=ik$ or $ik-2$ and let $P_{(G,n)} \in P(\Z G)$ be such that $[P_{(G,n)}] = \sigma_{ik}(G) \in C(\Z G)/T_G$. The classes $[P_{(G,n)}] \in C(\Z G)$ have two special properties. Firstly:
\begin{enumerate}
\item[(P1)] If $\theta \in \Aut(G)$, then $[(P_{(G,n)})_\theta] = [P_{(G,n)}] \in C(\Z G)/ T_G$.
\end{enumerate}

This was necessary in order to define our action of $\Aut(G)$ on the class $[P_{(G,n)}]$ since $\theta \in \Aut(G)$ sends $P \mapsto (I,\psi(\theta)) \otimes P_\theta$ for $P$ of rank one and some $\psi: \Aut(G) \to (\Z/|G|)^\times$. It is well-defined since 
\[ [(I,\psi(\theta)) \otimes P_\theta] = [(I,\psi(\theta))] + [P_\theta] = [P_{(G,n)}],\]
i.e. $[P_{(G,n)}] - [(P_{(G,n)})_\theta] \in T_G$. The second property follows from Theorem \ref{thm:main}:

\begin{enumerate} \setcounter{enumi}{\value{enumi}+1}
\item[(P2)] $[P_{(G,n)}]$ has cancellation if and only if $[P_{(G,n)}]/\Aut(G)$ has cancellation.
\end{enumerate}

The aim of this section will be to give examples to show that both properties are false for general projective modules over groups with periodic cohomology.

\subsection{Cyclic group of order $p$}

Here we will consider the case $G=C_p$ for a prime $p$. Recall that $T_{C_p} = 0$ \cite[Corollary 6.1]{Sw60a} and so we need only find $[P] \in C(\Z C_p)$ and $\theta \in \Aut(C_p)$ for which $[P] \ne [P_\theta] \in C(\Z C_p)$.

First define $i: \Z C_p \twoheadrightarrow \Z[\zeta_p]$ by mapping a generator $x \in C_p$ to $\zeta_p = e^{2 \pi i/p} \in \C$. It follows from a theorem of Rim \cite[Theorem 6.24]{Ri59} that the map
\[ i_* : C(\Z C_p) \to C(\Z [\zeta_p])\]
is an isomorphism. Let $\tilde{\cdot} :\Aut(C_p) \to \Gal(\Q(\zeta_p)/\Q)$ be the group isomorphism sending $\theta_i : x \mapsto x^i$ to $\tilde{\theta}_i : \zeta_p \mapsto \zeta_p^i$ for all $i \in (\Z/p)^\times$. The following is easy to check:

\begin{lemma} \label{lemma:rim-thm-commutes-with-aut}
If $\theta \in \Aut(C_p)$ and $[P] \in C(\Z C_p)$, then $i_*(\theta_*([P])) = \tilde{\theta}_*(i_*([P])$, i.e. there is a commutative diagram
\[
\begin{tikzcd}
C(\Z C_p) \ar[d,"\theta_*"] \ar[r,"i_*"] & C(\Z [\zeta_p]) \ar[d,"\tilde{\theta}_*"] \\
C(\Z C_p) \ar[r,"i_*"] & C(\Z [\zeta_p])
\end{tikzcd}
\]
\end{lemma}

Let $G = \Aut(C_p) \cong \Gal(\Q(\zeta_p)/\Q)$ and let $C(\Z C_p)^G$ and $C(\Z[\zeta_p])^G$ denotes the fixed points under the actions of $\Aut(C_p)$ and $\Gal(\Q(\zeta_p)/\Q)$ respectively. By Lemma \ref{lemma:rim-thm-commutes-with-aut}, we have that the map
\[i_*: C(\Z C_p)^G \to C(\Z[\zeta_p])^G\]
is an isomorphism, which could be viewed as an extension of Rim's theorem. 

\begin{prop} \label{prop:cyclic-example}
$C(\Z C_p)^G = C(\Z[\zeta_p])^G = 0$.
\end{prop}

\begin{proof}
It follows from the Chevalley's ambiguous class number formula \cite{Ch33} (see also \cite[Remark 6.2.3]{Gr03}) that
\[ |C(\Z[\zeta_p])^G| = \frac{\Ram(\Q(\zeta_p)/\Q)}{[\Q(\zeta_p) : \Q] \cdot [\Z^\times : \Z^\times \cap N_{\Q(\zeta_p)/\Q}(\Q(\zeta_p)^\times)]} \]
where $\Ram(\Q(\zeta_p)/\Q)$ is the product of the ramification indices at the finite and infinite places of $\Q(\zeta_p)$. It is easy to check that $\Ram(\Q(\zeta_p)/\Q)=2(p-1)$. To compute the denominator, note that $\Q(\zeta_p+\zeta_p^{-1})$ is the fixed field of $\Q(\zeta_p)$ under the conjugation action. If $\alpha \in \Q(\zeta_p)$, then the transitivity property of norms $N_{\Q(\zeta_p)/\Q} = N_{\Q(\zeta_p+\zeta_p^{-1})/\Q} \circ N_{\Q(\zeta_p)/\Q(\zeta_p+\zeta_p^{-1})}$ implies that
\[ N_{\Q(\zeta_p)/\Q}(\alpha) =N(\alpha \cdot \bar{\alpha}) = N(\alpha) \cdot N(\bar{\alpha}) = |N(\alpha)|^2 \ge 0\]
where $N = N_{\Q(\zeta_p+\zeta_p^{-1})/\Q}$.
This implies that there does not exist $\alpha \in \Q(\zeta_p)^\times$ such that $N_{\Q(\zeta_p)/\Q}(\alpha)=-1$ and so $[\Z^\times : \Z^\times \cap N_{\Q(\zeta_p)/\Q}(\Q(\zeta_p)^\times)]=2$. Hence $|C(\Z[\zeta_p])^G|=1$ and so $C(\Z C_p)^G = C(\Z[\zeta_p])^G = 0$ by the discussion above.
\end{proof}

Since $G$ acts trivially on the class $0 \in C(\Z C_p)$, Proposition \ref{prop:cyclic-example} implies that $G$ acts non-trivially on every class $[P] \ne 0 \in C(\Z C_p)$. This shows that $\Aut(C_p)$ acts non-trivially on $C(\Z C_p)$ if and only if $C(\Z C_p) \cong C(\Z[\zeta_p])$ is non-trivial. This is the case if and only if $p \ge 23$ by \cite{MM76}. Hence we have:

\begin{thm} \label{thm:example1}
If $G = C_p$ for $p \ge 23$ prime, then there exists $[P] \in C(\Z G)$ and $\theta \in \Aut(G)$ such that $[P_\theta] \ne [P] \in C(\Z G)/T_G$.
\end{thm}

More explicitly, let $\Aut(C_{23})=\{\theta_i : x \mapsto x^i \mid i \in (\Z/{23})^\times\}$ and recall that $C(\Z C_{23}) \cong \Z/3$ \cite[p30]{Mi71}. We know from above that the action of $\Aut(C_{23})$ on $C(\Z C_{23})$ is non-trivial. Since there is a unique map 
\[ \Aut(C_{23}) \cong \Z/22 \to \Aut(\Z/3) \cong \Z/2\]
with non-trivial image, the induces map $(\theta_i)_* : C(\Z C_{23}) \to C(\Z C_{23})$ has $(\theta_i)_*(1) =2$ if $i$ is odd and $(\theta_i)_*(1) =1$ if $i$ is even. Hence $|C(\Z C_{23}) /\Aut(C_{23})|=2$. This could also be checked directly by calculating the non-principal ideals in $C(\Z[\zeta_{23}])$.

\subsection{Quaternion group of order 28}

Note that $C(\Z Q_{28}) \cong \Z/2$ \cite[Theorem III]{Sw83} and so, since $\Aut(Q_{28})$ fixes $0$, it must act trivially on $C(\Z Q_{28})$, i.e. for all $[P] \in C(\Z Q_{28})$ and $\theta \in \Aut(Q_{28})$, we have $[P] = [P_\theta]$. Hence we can define an action of $\Aut(Q_{28})$ on each class $[P]$ by letting $\theta \in \Aut(Q_{28})$ send $P_0 \in [P]$ to $(P_0)_\theta \in [P]$. This is the only action of the required form since $T_{Q_{28}}=0$ \cite[Theorem VI]{Sw83}. In contrast to our second property, we show the following which is Theorem \ref{thm:main-example-intro} from the introduction.

\begin{thm} \label{thm:main-example} 
There exists a projective $\Z Q_{28}$-module $P$ such that $[P]$ has non-cancellation and $[P]/\Aut(Q_{28})$ has cancellation, where $\theta \in \Aut(Q_{28})$ acts on $[P]$ by sending $P_0 \mapsto (P_0)_\theta$.
\end{thm}

Our approach will be to use the action of $\Aut(G)$ on Milnor squares computed in \cite[Lemma 8.6]{Ni20a}, and our computations will be similar to the case $Q_{24}$ which was worked out in \cite[Section 9]{Ni20a}.

Fix the standard presentation $Q_{28} = \langle x, y \mid x^{7}=y^2, yxy^{-1}=x^{-1}\rangle$. Note that $(x^7+1)(x^7-1) = x^{14}-1 = 0$ implies that the ideals $I = (x^7-1)$ and $J= (x^7+1)$ have that $I \cap J = (0)$ and $I+J = (2,x^7-1)$. Let $\psi = x^6-x^5+x^4-x^3+x^2-x+1$ and note also that $x^7+1 = \psi \cdot (x+1)$ implies that the ideals $I' = (x+1)$ and $J' = (\psi)$ have $I' \cap J' = (x^7+1)$ and $I' + J' = (7,x+1)$.

By \cite[Example 42.3]{CR87}, we have the following two Milnor squares where $\Lambda = \Z Q_{28}/(x^7+1)$ and $\Z D_{14} \cong \Z Q_{28}/(x^7-1)$ with $D_{14}$ the dihedral group of order $14$.
\[
\begin{tikzcd}
	\Z Q_{28} \ar[r,"f_2"] \ar[d,"f_1"] & \Lambda \ar[d,"g_2"]\\
	\Z D_{14} \ar[r,"g_1"] & \F_2 D_{14}
\end{tikzcd} \quad
\begin{tikzcd}
	\Lambda \ar[r,"i_2"] \ar[d,"i_1"] & \Z Q_{28}/(\psi) \ar[d,"j_2"]\\
	\Z Q_{28}/(x+1) \ar[r,"j_1"] & \F_7 Q_{28}/(x+1)
\end{tikzcd}
\]

\begin{lemma} \label{lemma:f2=bij}
$(f_2)_* : \Cls(\Z Q_{28}) \to \Cls(\Lambda)$ is a bijection.
\end{lemma}

\begin{proof}
First note that $C(\Z D_{14}) \cong C(\Z[\zeta_{7}+\zeta_7^{-1}])$ \cite[p328]{RU74} and $C(\Z[\zeta_7+\zeta_7^{-1}]) \le C(\Z[\zeta_7]) = 1$ \cite{MM76}. Since $\Z D_{14}$ satisfies the Eichler condition, we have cancellation by Theorem \ref{thm:sj} and so $\Cls(\Z D_{14})=\{ \Z D_{14}\}$. By \cite[Theorem A10]{Sw83}, there is a surjection
\[ ((f_1)_*,(f_2)_*): \Cls(\Z Q_{28}) \twoheadrightarrow \Cls(\Z D_{14}) \times \Cls(\Lambda) \cong \Cls(\Lambda)\]
and, by the discussion above, the fibres are in bijection with
\[ \Z D_{14}^\times \backslash \F_2 D_{14}^\times \slash \Aut(P)\]
for $P \in \Cls(\Lambda)$. By combining \cite[Theorem A18, Lemma 10.1, Lemma 10.11]{Sw83}, we have
\[ \Z D_{14}^\times \backslash \F_2 D_{14}^\times \cong \frac{K_1(\F_2 D_{14})}{K_1(\Z D_{14})} \cong \frac{\F_2[\zeta_7+\zeta_7^{-1}]^\times}{\Z[\zeta_7+\zeta_7^{-1}]^\times} = 1\]
and so $\Z D_{14}^\times \backslash \F_2 D_{14}^\times \slash \Aut(P) =1$ and $(f_2)_*$ is bijective.	
\end{proof}

Recall that, for the standard presentation $Q_{28} = \langle x, y \mid x^{7}=y^2, yxy^{-1}=x^{-1}\rangle$, we have $\Aut(Q_{28}) =\{ \theta_{a,b}: a \in (\Z/{14})^\times, b \in \Z/{14}\}$ where $\theta_{a,b}(x)=x^a$ and $\theta_{a,b}(y) = x^{b}y$ as noted in Section \ref{section:orders-in-QG}. The following is easy to check:

\begin{lemma}
Let $a \in (\Z/{14})^\times$, $b \in \Z/{14}$. Then there exists $\tilde{\theta}_{a,b} \in \Aut(\Lambda)$ such that $f_2 \circ \theta_{a,b} = \tilde{\theta}_{a,b} \circ f_2$.
\end{lemma}

Now note that $\Z Q_{28}/(\psi) \cong \Z[\zeta_{14},j]$ and $\Z Q_{28} / (x+1) \cong \Z[j]$. This allows us to rewrite the second square as follows:
\[
\begin{tikzcd}
	\Lambda \ar[r,"i_2"] \ar[d,"i_1"] & \Z[\zeta_{14},j] \ar[d,"j_2"]\\
	\Z[j] \ar[r,"j_1"] & \F_7[j]
\end{tikzcd}
\quad
\begin{tikzcd}
	x,y \ar[r,mapsto] \ar[d,mapsto] & \zeta_{14},j \ar[d,mapsto]\\
	-1,j \ar[r,mapsto] & -1,j
\end{tikzcd}
\]

By \cite[Theorem A18]{Sw83}, we now have a commutative diagram with exact rows:
\[
\begin{tikzcd}
\Z[j]^\times \backslash \F_7[j]^\times \slash \Z[\zeta_{14},j]^\times \ar[d,twoheadrightarrow] \ar[r,hook] & \Cls(\Lambda) \ar[d,twoheadrightarrow] \ar[r,twoheadrightarrow] & \Cls(\Z[j]) \times \Cls(\Z[\zeta_{14},j]) \ar[d,twoheadrightarrow]\\
\frac{K_1(\F_7[j])}{K_1(\Z[j]) \times K_1(\Z[\zeta_{14},j])} \ar[r,hook] & C(\Lambda) \ar[r,twoheadrightarrow] & C(\Z[j]) \times C(\Z[\zeta_{14},j])
\end{tikzcd}
\]
where the left vertical map is induced by the map of units $\F_7[j]^\times \to K_1(\F_7[j])$.

Note that $\Cls(\Z[j]) = \{ \Z[j]\}$ since $\Z[j]$ is a PID and $\Cls(\Z[\zeta_{14},j]) = \{\Z[\zeta_{14},j]\}$ by \cite[p85]{Sw83}. By exactness, this implies that the two inclusions are bijections.

If $\mathcal{O}$ is a $\Z$-order in a finite-dimensional semisimple $\Q$-algebra with centre $R$, and $\nu: K_1(\mathcal{O}) \to R^\times$ is the reduced norm, then we define $SK_1(\mathcal{O}) = \Ker(\nu)$ \cite[p325]{MOV83}.

\begin{lemma} \label{lemma:dcoset1}
	$\frac{K_1(\F_7[j])}{K_1(\Z[j]) \times K_1(\Z[\zeta_{14},j])} \cong \F_7^\times/(\F_7^\times)^2$ where the map $K_1(\F_7[j]) \cong \F_7[j]^\times \to \F_7^\times$ is induced by the norm $N(a+bj) = a^2+b^2$.
\end{lemma}

\begin{proof}
Note that $K_1(\Z[j]) \cong \Z[j]^\times = \langle j \rangle$ since $\Z[j]$ is a Euclidean domain and $K_1(\F_7[j]) \cong \F_7[j]^\times$ since $\F_7[j]$ is a finite and hence semilocal ring. It follows from \cite[Lemma 7.6]{MOV83} that the map $N : \F_7[j]^\times \to \F_7^\times$ has kernel $SK_1(\Z[\zeta_{14},j])$. Since $\langle j \rangle \le \Ker(N)$, this implies that $N$ induces an isomorphism
\[ \frac{K_1(\F_7[j])}{K_1(\Z[j]) \times K_1(\Z[\zeta_{14},j])} \cong \F_7^\times/\left(\frac{K_1(\Z[\zeta_{14},j])}{SK_1(\Z[\zeta_{14},j])}\right) \cong \F_7^\times / \IM(\nu)\]
where $\nu : K_1(\Z[\zeta_{14},j]) \to \Z[\zeta_7+\zeta_7]^\times$ since $\Q[\zeta_{14},j]$ has centre $\Q(\zeta_7+\zeta_7^{-1})$. By \cite[Theorem 2 (ii)]{Wi77}, we have that $\IM(\nu) = (\Z[\zeta_7+\zeta_7^{-1}]^\times)^+$, the totally positive units. Since $|C(\Z[\zeta_7])|=1$ is odd, \cite[Corollary B25]{Sw83} implies that $(\Z[\zeta_7+\zeta_7^{-1}]^\times)^+ = (\Z[\zeta_7+\zeta_7^{-1}]^\times)^2$ and this has image $(\F_7^\times)^2$ in $\F_7^\times$.	
\end{proof}

Since the inclusion is a bijection, this implies that
\[ C(\Z Q_{28}) \cong C(\Lambda) \cong \F_7^\times/(\F_7^\times)^2 \cong \Z/2\]
by combining with Lemma \ref{lemma:f2=bij}.

\begin{lemma} \label{lemma:dcoset2}
$\Z[j]^\times \backslash \F_7[j]^\times \slash \Z[\zeta_{14},j]^\times = \{ [1],[1+j],[1+2j],[1+4j]\}$.	
\end{lemma}

\begin{proof}
By \cite[Lemma 7.5 (b)]{MOV83}, we have $\Z[\zeta_{14},j]^\times = \langle \Z[\zeta_{14}]^\times, j \rangle$. Now the map $\Z[\zeta_{14}]^\times \to \F_7[j]^\times$ maps $\zeta_{14} \mapsto 1$ and, for all $r \in \F_7^\times$, the cyclotomic unit $(\zeta_7^r-1)/(\zeta_7-1)$ maps to $r \in \F_7^\times$ and so this has image $\F_7^\times$. Now $\F_7[j] = \F_7[x]/(x^2+1)$ and $\left( \frac{-1}{7} \right)	=-1$ implies that $x^2+1$ is irreducible and so $\F_7[j]$ is a field and $\F_7[j]^\times = \F_7[j] \setminus \{0\}$. Since $\Z[j]^\times = \langle j \rangle$, this implies that 
\[ \Z[j]^\times \backslash \F_7[j]^\times \slash \Z[\zeta_{14},j]^\times = \F_7[j]^\times / \F_7^\times \cdot \langle j \rangle = (\F_7[j] \setminus \{0\}) / \F_7^\times \cdot \langle j \rangle.  \]
By acting by $\F_7^\times$ and $j$, each coset has a representative of the form $1+ aj$ for some $a \in \F_7$. If $a \ne 0$, then acting by $j$ shows that $[1+aj] = [1-a^{-1}j]$. Now $\{1,2,4\} \subseteq \F_7^\times$ represent the class under this action and so we get:
\[ (\F_7[j] \setminus \{0\}) / \F_7^\times \cdot \langle j \rangle = \{ [1],[1+j],[1+2j],[1+4j]\}. \qedhere \]
\end{proof}

By combining Lemmas \ref{lemma:dcoset1} and \ref{lemma:dcoset2}, the map 
\[ \Cls(\Z Q_{28}) \cong \Cls(\Lambda) \to C(\Z Q_{28}) \cong C(\Lambda)\] 
is given by
\[ N : \{ [1],[1+j],[1+2j],[1+4j]\} \to \F_7^\times/(\F_7^\times)^2.\]
Since $(\F_7^\times)^2 = \{1,2,4\}$ and $N(1)=1$, $N(1+j)=2$, $N(1+2j)=5$, $N(1+4j) = 3$, we have that
\[ \Cls^0(\Z Q_{28}) \leftrightarrow \{ [1], [1+j]\}, \quad \Cls^1(\Z Q_{28}) \leftrightarrow \{[1+2j],[1+4j]\}\]
where we identify $C(\Z Q_{28}) \cong \Z/2$.

Let $\mathcal{R}$ be the Milnor square defined above and let $\Aut(\mathcal{R})$ denote the set of Milnor square automorphisms (see \cite[Section 8]{Ni20a}).

\begin{lemma} \label{lemma:induced-action}
Let $a \in (\Z/{14})^\times$, $b \in \Z/{14}$. Then $\tilde{\theta}_{a,b} \in \Aut(\Lambda)$ extends to a Milnor square automorphism
\[ \hat{\theta}_{a,b} = (\tilde{\theta}_{a,b},\theta_{a,b}^1,\theta_{a,b}^2,\bar{\theta}_{a,b}) \in \Aut(\mathcal{R})\]
where $\bar{\theta}_{a,b} : j \mapsto (-1)^b j$.
\end{lemma}

It is easy to see that $\bar{\theta}_{a,b}$ fixes $[1]$ and $[1+j]=[1-j]$ for all $a \in (\Z/{14})^\times$, $b \in \Z/{14}$. However, we also have:
\[ \bar{\theta}_{a,b}([1+2j]) = 
\begin{cases}
 	[1+2j], & \text{if $b$ is even}\\
	[1-2j]=[1+4j], & \text{if $b$ is odd}
\end{cases}
\]
By \cite[Corollary 8.8]{Ni20a}, this implies that:
\[ \Cls^0(\Z Q_{28})/\Aut(Q_{28}) \leftrightarrow \{ [1], [1+j]\}, \quad \Cls^1(\Z Q_{28})/\Aut(Q_{28}) \leftrightarrow \{[1+2j]\}.\]
In particular, if $[P] = 1 \in \Z/2 \cong C(\Z Q_{28})$, then $[P]$ has non-cancellation but $[P]/\Aut(Q_{28})$ has cancellation and this completes the proof of Theorem \ref{thm:main-example}.

In fact, by applying \cite[Theorems 2.2 and 3.11]{BW08}, we can give explicit generators and relations for all projective modules involved:
\[ \Cls^0(\Z Q_{28}) \leftrightarrow \{\Z Q_{28}, (1+y,1+x)\}, \quad \Cls^1(\Z Q_{28}) = \{ (1+2y,1+x),(1+4y,1+x)\}\]
so that $[(1+2y,1+x)]$ has non-cancellation but $(1+2y,1+x)_{\theta_{1,1}} \cong (1+4y,1+x)$ and so $[(1+2y,1+x)]/\Aut(Q_{28})$ has cancellation.

\bibliography{biblio.bib}
\bibliographystyle{amsalpha}

\end{document}